\documentclass[12pt]{amsart}

\usepackage{amsmath}
\usepackage{amsthm}
\usepackage{xfrac}
\usepackage{txfonts}
\usepackage{indentfirst}
\usepackage{mathrsfs}
\usepackage{mathtools}
\usepackage{amssymb}
\usepackage{dutchcal}

\newtheorem{theorem}{Theorem}[section]
\newtheorem{lemma}{Lemma}[section]

\newtheorem{definition}{Definition}[section]

\newtheorem{proposition}{Proposition}[section]
\newtheorem{corollary}{Corollary}[section]

\newtheorem{claim}{Claim}
\newtheorem*{mainthm}{Main Theorem}
\newtheorem*{thmA}{Theorem A}

\newtheorem*{thmB'}{Theorem B'}
\newtheorem*{thmB''}{Theorem B''}

\numberwithin{equation}{section}

\newcommand{\R}{\mathbb{R}}
\newcommand{\Z}{\mathbb{Z}}
\newcommand{\eps}{\varepsilon}

\title{A Dichotomy for the dimension of SRB measure}
\author{Haojie Ren}

\date{\today}
\address{School of Mathematical Sciences, Fudan University, No 220 Handan Road, Shanghai, China 200433}
\email{20110180012@fudan.edu.cn}
\begin{document}
\setlength{\parindent}{1em}

\maketitle

\begin{abstract}
We study dynamical systems generated by skew products:
$$T: [0,1)\times\mathbb{R}\to [0,1)\times\mathbb{R}
\quad\quad T(x,y)=(bx\mod1,\gamma y+\phi(x))$$
where  integer $b\ge2$, $0<\gamma<1$ and $\phi$ is a real analytic $\mathbb{Z}$-periodic function.  We prove the following  dichotomy for the SRB measure $\omega$ for $T$: Either the support of $\omega$ is a graph of real analytic function, or the  dimension of $\omega$ is  equal to $\min\{2,1+\frac{\log b}{\log1/\gamma}\}$. Furthermore, given $b$ and $\phi$, the former alternative only  happens for finitely many $\gamma$ unless $\phi$ is constant.
\end{abstract}

\section{introduction}
In this paper, we consider dynamical systems generated by skew products:
\begin{equation}\label{Map}
T: [0,1)\times\mathbb{R}\to [0,1)\times\mathbb{R}
\quad\quad T(x,y)=(bx\mod1,\gamma y+\phi(x))
\end{equation}
where $b\ge2$ is an integer, $0<\gamma<1$ is a real number and $\phi$ is a non-constant $\mathbb{Z}$-periodic Lipschitz  function.
There exists an ergodic probability measure $\omega$ on $[0,1)\times\mathbb{R}$ such that almost every point $z\in [0,1)\times\mathbb{R}$ is generic, that is,
\begin{equation}\nonumber
\lim_{n\to\infty}\frac1n\sum_{k=0}^{n-1}\delta_{T^k(z)}=\omega\quad\text{weakly}.
\end{equation}
So $\omega$ is the unique SRB measure for $T$. Let $F^{\phi}_{b,\gamma}$ be the support of this measure, which is called the solenoidal attractor for $T$ (See \cite[Section 2]{tsujii2001fat}).

 A probability measure  $\mu$ in a metric space $X$ is called {\em exact-dimensional} if there exists a constant $\beta\ge 0$ such that for $\mu-\text{a.e.}$ $x$,
 \begin{equation}
 \lim_{r\to0}\frac{\log\mu\big(\mathbf{B}(x,r)\big) } {\log r}=\beta.
 \end{equation} 
 In this situation, we write $\text{dim}(\mu)=\beta$ and call it the dimension of $\mu$.
  For any set $K\subset\mathbb{R}^2$, let $\text{dim}_H(K)$ be the Hausdorff dimension of $K$. In the present paper, we study the dimension of the SRB measure $\omega$ and the Hausdorff dimension of the attractor $F^{\phi}_{b,\gamma}$. The main result of this paper is following.
 
 \begin{mainthm} Let $b\ge 2$ be an integer, $\gamma\in (0,1)$ and let $\phi$ be a $\Z$-periodic real analytic function. Then $\omega$ is exact dimensional and exactly
 	one of the following holds:
 	\begin{enumerate}
 		\item [(i)] $F^{\phi}_{b,\gamma}$ is a graph of a real analytic function;
 		\item [(ii)] $dim_H(F^{\phi}_{b,\gamma})=dim (\omega)=\min\{2,1+\frac{\log b}{\log1/\gamma}\}.$
 	\end{enumerate}
 	Moreover, given $b$ and  non-constant $\phi$, the first alternative only holds for finitely many $\gamma\in (0,1)$.
 \end{mainthm}
\medskip
{\bf Historical remarks.} In the work \cite{Alexander1984fat},  Alexander and Yorke considered a class maps called generalized baker's transformation:
\begin{equation}\nonumber
B: [-1,1]\times[-1,1]\circlearrowleft
\quad\quad B(x,y)=\begin{cases}
(2x-1,\gamma y+(1-\gamma))&\, x\ge0\\
(2x+1,\gamma y-(1-\gamma))&\, x<0.
\end{cases}
\end{equation}
Alexander and Yorke studied the case $\frac12<\gamma\le1$ where $T$ is locally area expanding. They showed that the map $B$ admits an absolutely continuous ergodic measure(ACEM) if and only if the number $\gamma$ satisfies a  condition: absolute continuity of the corresponding infinitely convolved Bernoulli measure. Erd\"os \cite{E} proved that $B$ admits no ACEM if $1/\gamma$ is Pisot number. On the other hand, $B$ admits an ACEM for Lebesgue almost every $\gamma\in (1/2,1]$ according to the results of Solomyak \cite{Solo,PS}. Later Shmerkin \cite{shmerkin2014} showed that the Hausdorff dimension of the exceptional set is zero.
Recently Varj\'u \cite{Varju2} show $B$ admits an ACEM  for a class of algebraic parameters.

To study the  class of dynamical systems that stably admits an ACEM with a negative Lyapunov exponents. Tsujii \cite{tsujii2001fat} introduced the class of dynamical systems generated by maps $T$, which is a generalization of the generalized baker's transformations $B$ from the view of smoothness(See (\ref{Map}) for the definition of $T$). 

In the case $b\gamma<1$ where $T$ contracts area, the SRB measure $\omega$ is totally singular with respect to the Lebesgue measure, thus the natural questions in this situation are what are the dimension of $\omega$ and the Hausdorff dimension of $F^{\phi}_{b,\gamma}$. Our paper gives a complete answer 
when $\phi$ is a real analytic $\mathbb{Z}$-periodic function.
For the case $b\gamma=1$ where $T$ preserves area, 
 our paper shows $\text{dim}(\omega)=2$ unless $F^{\phi}_{b,\gamma}$ is a graph of a real analytic function.
 For the case $b\gamma>1$, the natural questions are when $\omega$ is absolutely continuous with respect to the Lebesgue measure and what are the geometric properties of $F^{\phi}_{b,\gamma}$. In \cite{tsujii2001fat}, Tsujii proved that  the SRB measure is absolutely continuous respect to the Lebesgue measure for $C^2$ generic $\phi$. Later Avila, Gou\"ezel and Tsujii \cite{Avila} studied the smoothness of the SRB measure for $C^r$ generic $\phi$ ( for some integer $r\ge3$).  For the geometric properties of $F^{\phi}_{b,\gamma}$, Bam\'on et al. \cite{Bamon} proved the result: for any non-constant Lipschitz function $\phi(x)$ and integer $b\ge2$, there exists $\gamma_3\in(0,1)$ such that the  set $F^{\phi}_{b,\gamma}$ has non-empty interior for all $\gamma\in(\gamma_3,1)$.
  Our results show  weaker answers for these questions: 
  we figure out  when  $\text{dim}(\omega)\,\text{and}\,\text{dim}_{H}(F^{\phi}_{b,\gamma})$ are equal to 2  for the situation that $\phi$ is a real analytic periodic function. Note that, if $\omega$ is  absolutely continuous with respect to the Lebesgue measure, then $\text{dim}(\omega)=2.$

Another important reason to study the SRB measure $\omega$ is to study the Hausdorff dimension of the graph of Weierstrass-type functions
\begin{equation}\label{eqn:Wtype}\nonumber
W(x)=W^{\psi}_{\lambda,b}(x)=\sum\limits_{n=0}^{\infty}{{\lambda}^n\psi(b^nx)},\,\, x \in \mathbb{R}
\end{equation}
where integer $b > 1$, $1/b< \lambda < 1$ and $\psi(x):\mathbb{R} \to \mathbb{R}$ is a non-constant $\mathbb{Z}$-periodic Lipschitz function. 
The most famous example, with $\psi(x)=\cos(2\pi x)$, was introduced by Weierstrass as a continuous nowhere differentiable function, see \cite{hardy1916weierstrass}. 
Denote $\varGamma W^{\psi}_{\lambda,b}=\big\{\,\big(x,W^{\psi}_{\lambda,b}(x)\,\big)\,\big\}_{x\in[0,1]}$.
In fact Ledrappier \cite{ledrappier1992dimension} show that $\text{dim}_H(\varGamma W^{\psi}_{\lambda,b})$  is equal to $2+\frac{\log\lambda}{\log b}$ if $\text{dim}(\omega)=2$ where $\gamma=\frac1{b\gamma}$ and $\psi'(x)=\phi(x).$
Following this way, \cite{baranski2014dimension,shen2018hausdorff} studied the values of $\text{dim}_H(\varGamma W^{\psi}_{\lambda,b})$ when $\psi(x)=\cos(2\pi x)$.

 Let us also mention that Zhang \cite{Zhang2018} studied the smoothness of the SRB measures for the map  $T: [0,1)\times [0,1)\to  [0,1)\times [0,1)$ when $\gamma=1$. See \cite{Rams} for more general fat Baker maps.

\medskip{\bf Organization.} We shall introduce Theorem A and prove Main Theorem in Sect. \ref{ProveM}. The rest of the paper is devoted to prove Theorem A by Hochman's criterion on entropy increase (See Theorem \ref{thm:hochmanentgrow}). Thus we shall first recall Ledrappier-Young theorem and some basic properties of entropy of measures in Sect. \ref{pre}, then we shall analyze the  separation properties in Sect. \ref{sep}, entropy porosity in Sect. \ref{entropyporous} and transversality in Sect. \ref{tra}. In Sect. \ref{sec:partitionX} we will construct a nested sequence of partition of $\bigcup_{n=1}^{\infty} \{0,\ldots,b-1\} ^n$. Finally we  will assume the contrary and  use Hochman's criterion to obtain a contradiction in Sect. \ref{sec:pfThmA}.

\section{Main findings and proof of Main Theorem }\label{ProveM}
In this section, we will first introduce Theorem A, then we will give some explanations for the idea of the proof of Theorem A. Finally we will use Theorem A to finish the proof of the Main Theorem.

\subsection{Theorem A}
Let $\mathbb{Z}_+$ denote the set of positive integers. Let $\mathbb{N}$ be the set of nonnegative integers.
Let $\varLambda=\{ 0,1,...,b-1 \}$,
$\varLambda^{\#}=\bigcup_{n=1}^{\infty} \varLambda ^n$,
$\Sigma=\varLambda^{\mathbb{Z}_+}$. For any word $\textbf{j}=j_1j_2  \cdot \cdot \cdot j_p \in \varLambda^{p} $ of length $1\le p\le\infty$ and $x\in [0,1]$ define
\begin{equation}\label{S}
S(x,\textbf{j})=S_{\gamma, b}^\phi(x,\textbf{j})=\sum\limits_{n=1}^{p}{\gamma^{n-1}\phi\left(\frac x{b^n}+\frac{j_1}{b^n}+\frac{j_2}{b^{n-1}} + \cdot \cdot \cdot + \frac{j_n}b\right)},
\end{equation}
and the map
\begin{equation}
G : [0,1)\times\Sigma\to [0,1)\times\mathbb{R}\quad\quad G(x,\textbf{j})=(x, S(x,\textbf{j})).
\end{equation}
let $\nu$ denote even distributed probability  measure on $\varLambda$ and let $\nu^{\mathbb{Z}_+}$ be product measure on $\varSigma.$ Thus $G( \mathcal{m} \times\nu^{\mathbb{Z}_+})$ is the SRB measure $\omega$ and $$F^{\phi}_{b,\gamma}=\bigg\{(x, S(x,\textbf{j})):x\in [0,1),\,\textbf{j}\in\Sigma\bigg\}$$
where $\mathcal{m}$ is Lebesgue measure on $[0,1)$ (see \cite[Section 2]{tsujii2001fat} for details). 

In the work \cite{tsujii2001fat, Avila, baranski2014dimension, shen2018hausdorff}, $\omega$ has absolute continuity under  suitable transversality conditions for all $\gamma\in(1/b,1)$. In the following, for $\gamma\in(0,1)$ we shall recall a gentle transversality  the condition (H) and the degenerate situation the condition (H$^*$) from the work \cite{ren2021dichotomy}.
\begin{definition}
	Given an integer $b\ge 2$ and $\gamma\in (0,1)$,	we say that a $\mathbb{Z}$-periodic $C^1$ function $\phi(x)$ satisfies
	\begin{itemize}
		\item the condition (H) if 
		$$S(x,\textbf{j})-S(x,\textbf{i}) \nequiv 0, \quad \forall \, \textbf{j} \neq \textbf{i} \in \Sigma.$$
		\item
		the condition (H$^*$) if
		$$S(x,\textbf{j})-S(x,\textbf{i}) \equiv 0, \quad \forall \, \textbf{j},\,\, \textbf{i} \in \Sigma.$$
	\end{itemize}
\end{definition}
The following is our main result in this work.
\begin{thmA}
	If a real analytic $\mathbb{Z}$-periodic function $\phi(x)$ satisfies the condition (H) for an integer $b\ge 2$ and $\gamma\in (0,1)$, then
	$$dim (\omega)=\min\{1+\frac{\log b}{\log1/\gamma},2\}.$$
\end{thmA}
The idea of the proof of Theorem A is from the Hochman \cite{hochman2014self}'s breakthrough observation for entropy growth of measures under convolution. More specifically, a self-similar measure  is the convolution of  a  measure  with itself scaled down by some positive value, which allows Hochman to apply his criterion on entropy increase to get the dimension of self-similar measures in $\mathbb{R}$ under the exponential separation condition. Later the methods are generalized to study the self-affine measures on the plane \cite{barany2019hausdorff, Hochman2022}. Recently, following  \cite{barany2019hausdorff}, Shen and the author \cite{ren2021dichotomy} studied a class of measures induced by nonlinear IFS.
Our paper will take the similar strategy in \cite{barany2019hausdorff, ren2021dichotomy}. 

Indeed, for any  $x\in [0,1]$ define the map
\begin{equation}\label{S_x}
S_x: \Sigma\to\mathbb{R}\quad\quad S_x(\,\textbf{j}\,)=S(x,\textbf{j})
\end{equation}
and measure $m_x:=S_x(\nu^{\mathbb{Z}_+}).$ 
Note that $\omega=\int_{[0,1) }\big(\delta_x\times m_x\big)\,dx.$
By Ledrappier-Young theory \cite{ledrappier1985metric},
$m_x$ is exact dimensional and there exists a constant $\alpha\in[0,1]$ such that $\text{dim}(m_x)=\alpha$ for Lebesgue almost every $x\in [0,1]$, thus we only need to show $\alpha=\min\{1,\frac{\log b}{\log1/\gamma}\}$, see Sect \ref{ledrappier-young}. For any $n\in\mathbb{Z}_+$, $\textbf{j}=j_1j_2\ldots j_n\in\varLambda^n$ and $x\in[0,1]$, let
\begin{equation}\label{eq:text}
\textbf{j}(x)=\frac{x+j_1+\ldots+j_nb^{n-1}}{b^n}
\end{equation}
and
\begin{equation}\label{T^nm_x}
T^n(m_{\textbf{j}\,(x)})=f_{x,\,\textbf{j}}(m_{\textbf{j}(x)})
\end{equation}
where $f_{x,\textbf{j}}(y)=\gamma^ny+S(x,\,\textbf{j}),\,\forall y\in\mathbb{R}$.
By \cite[(9)]{tsujii2001fat} the following holds 
\begin{equation}\label{FundementalFormular}
m_x=\frac1{b^n}\sum_{\textbf{j}\in\varLambda^n} T^n(m_{\textbf{j}(x)})
\end{equation}
for $n\ge1$ and $x\in [0,1]$.
 Since most of $m_{\textbf{j}(x)}$ are similar in the sense of entropy when $n$ is large enough, see Sect \ref{entropyporous}.
Thus $m_x$ is approximate with
$$\bigg(\frac1{b^n}\sum_{\textbf{j}\in\varLambda^n}\delta_{S(x,\,\textbf{j})}\bigg)*(\gamma^nm_x).$$
Assuming the contrary, we shall apply the Hochman's criterion   on entropy growth \cite{hochman2014self} to obtain a contradiction.

In  \cite{hochman2014self,baranski2014dimension} they studied the linear systems and relative self-similar(affine) measures. But since $T$ is a nonlinear map, we shall 
use some ideas in \cite{ren2021dichotomy} to construct the convolution of measures by transversality.
In \cite{ren2021dichotomy} the authors considered the functions
$$S_{\textbf{j}}:[0,1]\to\mathbb{R}\quad S_{\textbf{j}}(x)=S(x,\textbf{j})$$
for any $\textbf{j}\in\Sigma.$ But in our paper we shall consider the functions $S_x$  which  is defined in symbolic space $\Sigma$  for any $x\in [0,1]$ (See (\ref{S_x}) for the definition of $S_x$). Thus we shall take a different way to analyze the separation, transversality and  construct a different partition for  symbolic space, which are important in this strategy to prove Theorem A. 

\subsection{Proof of Main Theorem}
To prove our main result, we shall need the dichotomy between condition (H) and condition (H$^*$). The following is a immediate consequence of \cite[Theorem 2.1]{gao2022}.
For the case $b\gamma>1$, see also \cite[Theorem A]{ren2021dichotomy}.
\begin{theorem}\label{thm:dichotomy}
	Fix $b\ge 2$ integer and $\gamma\in (0,1)$. Assume that $\phi$ is analytic $\Z$-periodic function. Then exactly one of the following holds:
	\begin{enumerate}
		\item [(i)]  $\phi$ satisfies the condition (H$^*$);
		\item [(ii)]  $\phi$ satisfies the condition (H).
	\end{enumerate}
\end{theorem}
\begin{proof}
	Let $\phi^*:\mathbb{R}\to\mathbb{R}$ be a function such that 
	$\phi^*(x)=\phi(x)-\int_{[0,1]}\phi(s)\,ds$ for $x\in\mathbb{R},$
	then 
	$\hat{f}(x):=\int_0^x\phi^*(s)\,ds,\,\forall x\in\mathbb{R}$ is an real analytic $\mathbb{Z}$-periodic function. For any $\textbf{w}\in\Sigma$, we have
	\begin{equation}\nonumber
	\int^x_0S^{\phi^*}_{b,\,\gamma}(s,\textbf{w})\,ds=\frac1\gamma\sum_{n=1}^{\infty}(b\gamma)^n\bigg(\hat{f}\circ\tau_{\textbf{w},\,n}(x)-\hat{f}\circ\tau_{\textbf{w},\,n}(0)\bigg)
	\end{equation}
	where $\tau_{\textbf{w},\,n}(x)=\frac{x+w_1+\ldots+w_nb^{n-1}}{b^n}.$
	Therefore by \cite[Theorem 2.1]{gao2022}, the following are equivalent:
	\begin{enumerate}
		\item [(i)]  there exists $\textbf{i}\neq\textbf{j}\in\varSigma$ with
		$ \int^x_0S^{\phi^*}_{b,\,\gamma}(s,\textbf{i})\,ds \equiv \int^x_0S^{\phi^*}_{b,\,\gamma}(s,\textbf{j})\,ds$;
		\item [(ii)] $\int^x_0S^{\phi^*}_{b,\,\gamma}(s,\textbf{i})\,ds \equiv \int^x_0S^{\phi^*}_{b,\,\gamma}(s,\textbf{0})\,ds$ for any $\textbf{j}\in\varSigma$.
	\end{enumerate}
	where $\textbf{0}=00\ldots0\ldots\in\Sigma.$ 
	Thus, since for every $\textbf{i},\,\textbf{j}\in\Sigma$,  the following are equivalent:
	\begin{enumerate}
		\item [(i)] $ \int^x_0S^{\phi^*}_{b,\,\gamma}(s,\textbf{i})\,ds \equiv \int^x_0S^{\phi^*}_{b,\,\gamma}(s,\textbf{j})\,ds$;
		\item[(ii)]$ S^{\phi}_{b,\,\gamma}(x,\textbf{i}) \equiv S^{\phi}_{b,\,\gamma}(x,\textbf{j}),$
	\end{enumerate}
	the Theorem \ref{thm:dichotomy} holds.
\end{proof}
\begin{proof}[Proof of the Main Theorem] 
	There exists $\gamma_1\in(0,1)$ such that
	$\omega$ is absolute continuous with respect with Lebsgue measure for all $\gamma\in(\gamma_1,1)$ (See the proof of Theorem 1.1 in \cite{shen2018hausdorff} for this conclusion). Thus we only need to consider the case $\gamma\in(0,\gamma_1)$. The proof for the dimension of SRB measure $\omega$ is  similar to the proof of \cite[Main Theorem]{ren2021dichotomy} by Theorem A and Theorem \ref{thm:dichotomy}. For non-degenerate $\gamma$, by the mass distribution principle, we have $dim_H(F^{\phi}_{b,\gamma})\ge\min\{2,1+\frac{\log b}{\log1/\gamma}\}$ since $F^{\phi}_{b,\gamma}$  is the support of $\omega$. Also we have $dim_B(F^{\phi}_{b,\gamma})\le\min\{2,1+\frac{\log b}{\log1/\gamma}\}$ by the definition of $F^{\phi}_{b,\gamma}$. Thus our theorem holds.
\end{proof}

\section{Preliminaries}\label{pre}
In this section we shall first recall some basic facts from Ledrappier-Young theory. Later we will recall the definition and basic properties of the entropy of measures, which is a basic tool in our paper.
\subsection{Ledrappier's Theorem}\label{ledrappier-young}
The following Theorem \ref{TheoremB} can be proved by similar methods in \cite[Proposition 2]{ledrappier1992dimension} (See \cite{ledrappier1985metric, Shu2010} for more general results on compact manifolds). For the completeness, we offer a proof in Appendix.
\begin{theorem}\label{TheoremB}
	If $\phi:\mathbb{R}\to\mathbb{R}$ is a $\mathbb{Z}$-periodic Lipschitz function, then
		\begin{enumerate}
		\item[(1)] $\omega$ is exact dimensional;
		\item[(2)] there is a constant $\alpha\in [0,1]$ such that for Lebesgue a.e. $x\in [0,1]$, $m_x$ is exact dimensional and $\dim(m_x)=\alpha$.
		\item[(3)]
		\begin{equation}
		\label{LedrapperYoungF}
		\dim(\omega)=1+\alpha.
		\end{equation}
	\end{enumerate}
\end{theorem}
So it suffices to show that $\alpha=\min\{1,\frac{\log b}{\log1/\gamma}\}$ under the condition (H) for proving Theorem A.

\subsection{Entropy of measures}
For a probability space $(\Omega, \mathcal{B}, \mu)$, a {\em countable partition} $\mathcal{Q}$ is a countable collection of pairwise disjoint measurable subsets of $\Omega$ whose union is equal to $\Omega$. Let $\mathcal{Q}(x)$  be the member of $\mathcal{Q}$ that contains $x$. For $\mu(\mathcal{Q}(x))>0$, we call the conditional measure
$$\mu_{\mathcal{Q}(x)}(A)=\frac{\mu(A\cap \mathcal{Q}(x))}{\mu(\mathcal{Q}(x))}$$
a {\em $\mathcal{Q}$-component} of $\mu$.
Define the {\em entropy}
$$H(\mu, \mathcal{Q})=\sum_{Q\in\mathcal{Q}} -\mu(Q) \log_b \mu(Q)$$
where the common convention $0\log 0=0$ is adopted.
For another countable partition $\mathcal{P}$, define the {\em condition entropy} as
$$H(\mu, \mathcal{Q}|\mathcal{P})=\sum_{P\in \mathcal{P},\,\mu(P)>0} \mu(P) H(\mu_P, \mathcal{Q}).$$
When $\mathcal{Q}$ is a {\em refinement} of $\mathcal{P}$, i.e.,
$\mathcal{Q}(x)\subset\mathcal{P}(x)$ for every $x\in \Omega$, we have
$$H(\mu, \mathcal{Q}|\mathcal{P})=H(\mu, \mathcal{Q})-H(\mu, \mathcal{P}).$$

If there exists a sequence of partitions $\mathcal{Q}_i$, $i=0,1,2,\cdots$, such that $\mathcal{Q}_{i+1}$ is a refinement of $\mathcal{Q}_i$, we shall denote $\mu_{x,i}=\mu_{\mathcal{Q}_i(x)}$, and call it a {\em $i$-th component measure of } $\mu$. For a finite set $I$ of integers, if for every $i\in I$, there is a random variable $Y_i$ defined over $(\Omega, \mathcal{B}(\mathcal{Q}_i), \mu)$, where $\mathcal{B}(\mathcal{Q}_i)$ is the sub-$\sigma$-algebra of $\mathcal{B}$ which is generated by $\mathcal{Q}_i$. Then we shall use the following notation
$$\mathbb{P}_{i\in I} (K_i)=\mathbb{P}_{i\in I}^\mu(K_i):=\frac{1}{\# I} \sum_{i\in I} \mu(K_i),$$
where $K_i$ is an event for $Y_i$. If $Y_i$'s are $\mathbb{R}$-valued random variable, we shall also use the notation
$$\mathbb{E}_{i\in I} (Y_i)=\mathbb{E}^\mu_{i\in I}(Y_i):=\frac{1}{\# I} \sum_{i\in I} \mathbb{E}(Y_i).$$
Therefore the following holds
$$H(\mu, \mathcal{Q}_{m+n}|\mathcal{Q}_n)=\mathbb{E}(H(\mu_{x, n}, \mathcal{Q}_{m+n}))=\mathbb{E}_{i=n} (H(\mu_{x,i},\mathcal{Q}_{i+m})).$$
These notations were used extensively in ~\cite{hochman2014self} and \cite{barany2019hausdorff}.

In most of cases, we shall consider the situation  that $\Omega=\mathbb{R}$ and $\mathcal{B}$ the Borel $\sigma$-algebra. Let $\mathcal{L}_n$ be the partition of $\mathbb{R}$ into $b$-adic intervals of level $n$, i.e., the intervals $[j/b^n, (j+1)/b^n)$, $j\in \mathbb{Z}$. Let $\mathscr{P}(\R)$ denote the collection of all Borel probability measures in $\R$. 
If an probability measure $\mu\in \mathscr{P}(\mathbb{R})$ is exact dimensional, its dimension is closely related to the entropy. More specifically we have the
 following fact
\cite[Theorem 4.4]{Young}. See also~\cite[Theorem 1.3]{Fan2002}.
\begin{proposition}\label{prop:Young}
	If $\mu \in \mathscr{P}(\mathbb{R})$ is exact dimensional, then
	$$\dim(\mu)= \lim\limits_{n\to\infty}\frac{1}{n} H(\mu,\mathcal{L}_n).$$
\end{proposition}

 The following are some well-known facts about entropy and conditional entropy, which will  be used a lot in our work.
See \cite[Section 3.1]{hochman2014self} for details.

\begin{lemma}[Concavity]\label{lem:concave}
	Consider a measurable space $(\mu, \mathcal{B})$ which is endowed with partitions $\mathcal{Q}$ and $\mathcal{P}$ such that $\mathcal{P}$ is a refinement of $\mathcal{Q}$. Let $\mu, \mu'$ be probability measures in $(\mu, \mathcal{B})$. The for any $t\in (0,1)$,
	$$tH(\mu,\mathcal{Q})+(1-t)H(\mu',\mathcal{Q})\le H(t\mu+(1-t)\mu',\mathcal{Q}),$$
	$$tH(\mu,\mathcal{P}|\mathcal{Q})+(1-t)H(\mu',\mathcal{P}|\mathcal{Q})\le H(t\mu+(1-t)\mu',\mathcal{P}|\mathcal{Q}).$$
\end{lemma}

\begin{lemma} \label{lem:affinetransform}
	Let $\mu\in \mathcal{P}(\mathbb{R})$. There is a constant $C>0$ such that for any affine map  $f(x)=ax+c$, $a, c\in \R$, $a\not=0$ and for any $n\in\mathbb{N}$ we have
	$$\left|H(f\mu,\,\mathcal{L}_{n+[\log_b |a|]})-H(\mu,\,\mathcal{L}_{n})\right|\le C.$$
\end{lemma}
%

\begin{lemma}\label{lem:pfclose}
	Given a probability space $(\Omega, \mathcal{B}, \mu)$, if $f,g:\Omega\to\mathbb{R}$ are measurable and $\sup_x|f(x)-g(x)|\le b^{-n}$ then
	$$\left|H(f\mu, \mathcal{L}_{n})-H(g\mu, \mathcal{L}_{n})\right|\le C,$$
	where $C$ is an absolute constant.
\end{lemma}

\section{Exponential separation}\label{sep}
In this section, we deduce from the condition (H) the  exponential separation properties, which is used in the section \ref{FunctionEntropy} to prove Theorem A. The standard method and  Definition \ref{def:ex} are from \cite{hochman2014self}. Before the statement of the theorem, it is convenient to introduce the following notation.

{\bf Notation.} For every integer $n\in\mathbb{N}$, let $\hat{n}$ be the unique integer such that
\begin{equation}\label{eqn:n'}
\gamma^{\hat{n}}\le b^{-n}< \gamma^{\hat{n}-1}.
\end{equation}In fact the following definition is a little different from \cite{hochman2014self}, which is more convenient in this paper.
\begin{definition}\label{def:ex}
	Let $E_1,E_2,\ldots$ be  subsets of $\mathbb{R}$. 
	For any $\eps>0$ and  $Q\subset\mathbb{Z}_+$,
	we say that the sequence $(E_n)_{n\in\mathbb{Z}_+}$ is $(\varepsilon,Q)$ -exponential separation if
	$$|p-q|>\varepsilon^{\hat{n} }\quad\quad \forall p\neq q\in E_{\hat{n}}$$
	for each $n\in Q.$
\end{definition}
For $\textbf{u}=u_1u_2\ldots u_t\in\varLambda^{\#}$ and $\textbf{j}\in\varLambda^{\#}\cup\Sigma$, let $\textbf{u}\textbf{j}=u_1u_2\ldots u_tj_1j_2\ldots\in\varLambda^{\#}\cup\Sigma$ as usual.
The main result of this section is the following Theorem \ref{thmSeperation}.
\begin{theorem}\label{thmSeperation}
	If $\phi(x)$ is a real analytic $\mathbb{Z}$-periodic function which satisfies the condition (H) for some integer $b\ge2$ and $\gamma\in(0,1)$, then there exists $\ell_0\in\mathbb{N}$ and $\varepsilon_0>0$ such that the following holds for  Lebesgue-a.e.$\,x\in [0,1]$. 
	 
	For any integer $\ell\ge\ell_0$, there exists a set $Q_{x,\,\ell}\subset\mathbb{Z}_+$ such that 
	\begin{enumerate}
		\item [(i)] $\#Q_{x,\,\ell}=\infty $;
		\item [(ii)] for any $\textbf{w}\in\varLambda^{\ell}$, the sequence $(X_{n}^{\textbf{w},\,x})_{n\in\mathbb{Z}_+}$ is $(\varepsilon_0, Q_{x,\,\ell})$-exponential separation
	\end{enumerate}
	where $X_{n}^{\textbf{w},\,x}=\big\{S(x,\,\textbf{j} \textbf{w} )\,:\,\textbf{j}\in\varLambda^{n-\ell}\big\}$ for $n>\ell$ and $X_{n}^{\textbf{w},\,x}=\{0\}$ for $n\le\ell$.
\end{theorem}
Before  giving the proof, it is necessary to introduce the following formulas. For any $x\in [0,1]$, $\textbf{w}=w_1w_2\ldots w_m\in\varLambda^{\#}$ and each $\textbf{i},\,\textbf{j}\in\varLambda^{\#}\cup\varSigma,$ 
we have
\begin{equation}\label{F}
S(x,\textbf{w}\textbf{i})=S(x,\textbf{w})+\gamma^{m}S(\textbf{w}(x),\textbf{i})
\end{equation}
by the definition of function $S(\cdot,\cdot)$ where  $\textbf{w}(x)$ is from (\ref{eq:text}).
This implies, for any $k\in\mathbb{N}$,
\begin{equation}\label{FM}
S^{(k)}(x,\,\textbf{w}\textbf{i})-S^{(k)}(x,\,\textbf{w}\textbf{j})=(\frac{\gamma}{b^k})^m\bigg(S^{(k)}(\textbf{w}(x),\,\textbf{i})-S^{(k)}(\textbf{w}(x),\,\textbf{j})\bigg).
\end{equation}

We also need to recall  the following  results in \cite[Lemma 5.8]{hochman2014self} and \cite[Lemma 5.2]{ren2021dichotomy}.
\begin{lemma}\label{lemHochmanSeperation}
Let $k\in\mathbb{N}$, and let F a k-times continuously differentiable function on a compact interval $J\subset\mathbb{R}$. Let $M=\parallel F\parallel_{J,k}$, and let $0<d<1$ be such that for every $x\in J$ there is a $p\in\{0,1,\ldots,k\}$ with $|F^{(p)}(x)|>d.$ Then for every $0<\rho<(d/2)^{2^k}$, the set $F^{-1}(-\rho,\rho)\subset F$ can be cover by $O_{k, M,|J|}(1/{d^k})$ intervals of length $\le2(\rho/d)^{1/{2^k}}$ each.
\end{lemma}
\begin{lemma}\label{lemRenT}
	If the condition (H) holds, then there exists a constant $\varepsilon_1>0$ and an integer $Q_1\ge0$ such that for every $\textbf{i},\,\textbf{j}\in\Sigma$ such that $i_1\neq j_1$ and any $x\in [0,1]$, there exists $k\in\{0,1,\ldots,Q_1\}$ such that
	$$|S^{(k)}(x,\,\textbf{i})-S^{(k)}(x,\,\textbf{j})|\ge \varepsilon_1.$$
\end{lemma}
The proof of \cite[Lemma 5.2]{ren2021dichotomy}   works for all $\gamma\in(0,1).$
For any $\textbf{u}=u_1u_2\ldots u_m\in\varLambda^{\#}$, every $\textbf{i}\neq\textbf{j}\in\varLambda^{\#}\cup\Sigma$, let 
$|\textbf{u}|=m$ and call $m$ the length of $\textbf{u}$. Let
$\textbf{i}\wedge\textbf{j}=i_1i_2\ldots i_n$ 
where $n+1$ is the smallest number  k such that $i_k\neq j_k. $
Let $\sigma:\Sigma\to\Sigma$ denote the shift map $(i_1i_2\ldots)\mapsto(i_2i_3\ldots)$.
We shall use Lemma \ref{thmSeperation} and Lemma \ref{lemRenT} to show that the exceptional parameters $x\in [0,1]$ in Theorem \ref{thmSeperation} has Lebesgue measure 0.
\begin{proof}[The proof of Theorem \ref{thmSeperation}]
	Let  $Q_1\in\mathbb{N}$ and $\varepsilon_1>0$ be the values in Lemma \ref{lemRenT}. 
	Let  $\ell_0$ be a fixed integer such that 
	$$\frac{(\gamma/b^{q})^{\ell_0}}{1-\gamma/b^{q}}\parallel\phi^{(q)}\parallel_{\infty}<\varepsilon_1/4$$
	for every $q\in\{0,1,\ldots,Q_1\}$.
	Let $\ell,n $ be the integers such that $\ell\ge\ell_0,\,n>\ell$. Let $\textbf{w}\in\varLambda^{\ell}$ and  $\textbf{i}\neq\textbf{j}\in\varLambda^{n-\ell}$, we shall consider the function $f^{\,\textbf{w}}_{\textbf{i},\,\textbf{j}}:[0,1]\to\mathbb{R}$
	$$f^{\,\textbf{w}}_{\textbf{i},\,\textbf{j}}(x)=S(x,\,\textbf{i}\textbf{w})-S(x,\,\textbf{j}\textbf{w})\quad\quad\forall x\in[0,1].$$
    By (\ref{FM})
	 for each $q\in\{0,1,\ldots,Q_1\}$ and $x\in[0,1]$, we have
	\begin{equation}\label{eq:exminus}
	(f^{\textbf{w}}_{\textbf{i},\,\textbf{j}})^{(q)}(x)=(\gamma/b^q)^{|\textbf{u}|}\big( S^{(q)}(\textbf{u}(x),\,\textbf{i}'\textbf{w})-S^{(q)}(\textbf{u}(x),\,\textbf{j}'\textbf{w})\big)
	\end{equation}
		where  $\textbf{u}=\textbf{i}\wedge\textbf{j}$ and $\textbf{i}'=\sigma^{|\textbf{u}|}\textbf{i},\,\textbf{j}'=\sigma^{|\textbf{u}|}\textbf{j}$.
		Also by  Lemma \ref{lemRenT}, there exists $k=k(x,\textbf{w},\,\textbf{i},\,\textbf{j})\in\{0,1,\ldots,Q_1\}$ such that
		$$\big|S^{(q)}(\textbf{u}(x),\,\textbf{i}'\textbf{w}\textbf{0}_{\infty})-S^{(q)}(\textbf{u}(x),\,\textbf{j}'\textbf{w}\textbf{0}_{\infty})\big|\ge\eps_1$$
		where $\textbf{0}_{\infty}=00\ldots0\ldots\in\Sigma.$
		Combining this with (\ref{eq:exminus}) and the definition of $\ell_0$, then the following holds since $\textbf{w}\in\varLambda^{\ell}$ and $\ell\ge\ell_0$.
	  For each $x\in[0,1]$  there exists $k=k(x,\textbf{w},\,\textbf{i},\,\textbf{j})\in\{0,1,\ldots,Q_1\}$ such that
     $$|(f^{\textbf{w}}_{\textbf{i},\,\textbf{j}})^{(k)}(x)|\ge b_n$$
     where 
	$b_n=\frac{\varepsilon_1}2(\gamma/{b^{Q_1}})^n$.
	
    Therefore we have the following property  by Lemma \ref{lemHochmanSeperation}.      For  every $0<\eps\le (\gamma^2/{b^{Q_1}})^{2^{Q_1}}$ let $n$ be an  integer such that $n>\ell$ and $ \eps^n <({b_n}/2)^{2^{Q_1}}$, then
      the set 
	$$E_{\varepsilon,\,n,\,\ell}=\bigcup_{\textbf{w}\in\varLambda^{\ell}}\bigcup_{\textbf{i}\neq\textbf{j}\in\varLambda^{n-\ell}}(f^{\textbf{w}}_{\textbf{i},\,\textbf{j}})^{-1}(-\varepsilon^n,\varepsilon^n)$$
    can be covered by $O(b^{2n-\ell}/{(b_n)^{Q_1}})$ intervals of length $\le2({\varepsilon^n}/{b_n})^{1/{2^{Q_1}}}$. By the above for any integer $N$ such that $\hat{N}>\ell$, we have
	$$\overline{dim_b}\bigg(\bigcap_{n\ge N}E_{\varepsilon,\,\hat{n},\,\ell}\bigg)\le\lim_{n\to\infty}\frac{\log\big(b^{2\hat{n}-\ell}/b_{\hat{n} }^{Q_1} \big)} {\log\big((b_{\hat{n}}/\varepsilon^{\hat{n}})^{1/2^{Q_1}}\big)}=\frac{2^{Q_1}\log(b^{2+Q_1^2}/\gamma^{Q_1})}{\log(\gamma/(b^{Q_1}\varepsilon))}.$$
	Let $\varepsilon_0=\varepsilon_0(\phi,\gamma,b,Q_1)>0$ be small enough such that
    $$\frac{2^{Q_1}\log(b^{2+Q_1^2}/\gamma^{Q_1})}{\log(\gamma/(b^{Q_1}\varepsilon))}<1,$$
     thus $\mathcal{m}(\bigcap_{n\ge N}E_{\varepsilon,\,\hat{n},\,\ell})=0$ for all  integer $N$ s.t. $\hat{N}>\ell$.
     
     Denote $E_{\varepsilon_0}=\bigcup_{\ell\ge\ell_0}\bigcup_{\hat{N}>\ell }\bigcap_{n\ge N}E_{\varepsilon_0,\,\hat{n},\,\ell}$. Therefore, for any $x\in [0,1]\setminus E_{\varepsilon_0}$ and $\ell\ge\ell_0$, there exists $Q_{x,\ell}$ such that (i) and (ii) holds.
	\end{proof}

\section{Entropy Porosity}\label{entropyporous}
In this section we shall analyze the entropy porosity of measures $m_x$ under the condition(H). This property will be used  to obtain entropy growth under convolution by Hochman's  criterion ( See Theorem \ref{thm:hochmanentgrow}) in Subsection \ref{proveA}.

\begin{definition}[Entropy porous]
A measure $\mu \in \mathscr{P}(\mathbb{R})$ is $(h,\delta,m)$-\em{ entropy porous} from scale $n_1$ to $n_2$ if
	$$\mathbb{P}^{\,\mu}_{n_1\le i \le n_2} \left(\frac{1}{m} H (\mu_{x,i},\mathcal{L}_{i+m})<h+\delta\right) >1-\delta.$$
\end{definition}

The main result of this section is the following Theorem~\ref{thm:entporous}. 




\begin{theorem}\label{thm:entporous}
	Fix an integer $b\ge 2$ and $\gamma\in (0,1)$. Assume that $\phi:\mathbb{R} \to \mathbb{R}$ is a $\mathbb{Z}$-periodic Lipschitz function such that condition (H) holds. Then
	for any $\eps>0$, $m\ge M_1(\eps),$ $k\ge K_1(\eps,m)$ and $n\ge N_1(\eps,m,k)$, the following holds:	
	$$
	\nu^{n}\left(\left\{\textbf{i}\in \varLambda^{n}: m_{\textbf{i}(0)} \text{ is } (\alpha, \eps, m)-\text{entropy porous from scale } 1 \text{ to } k\right\}\right)>1-\eps
	$$
	where $\alpha$ is a constant in Theorem \ref{ledrappier-young}.
\end{theorem}
Recall that $\textbf{i}(0)=\frac{i_1+i_2b+\dots+i_nb^{n-1}}{b^n}$ for $\textbf{i}\in \varLambda^{n}$. We shall follow the strategy in \cite[Section 3]{barany2019hausdorff} to prove  Theorem \ref{thm:entporous} (See or \cite[Section 4]{ren2021dichotomy}). The only difficulty is to show $m_x$ is nonatomic.

\subsection{Nonatomic measure $m_x$} This subsection is devoted to prove the 
 measure $m_x$ is nonatomic under the condition (H), which will be used to  show  the uniform continuity across scales of the measures $m_x$ (See the next subsection for the definition).
\begin{lemma}\label{lemNonatomic}
	If the condition (H) holds, the measure $m_x$ has no atom for any $x\in [0,1].$
\end{lemma}
\begin{proof}
	By the definition of the measure $m_x$, there exists a constant $L=L(\phi,b,\gamma)>0$ such that $m_x$ is support in $[-L,L]$ for all $x\in[0,1]$.
	If the lemma fails, 
	these exists $x_1\in [0,1]$ and $p\in\mathbb{R}$ such that 
	\begin{equation}\label{EqMaxActom}
	m_{x_1}(\{p\})=\max_{x\in [0,1],\,z\in\mathbb{R}}m_x(\{z\})>0
	\end{equation}
	by the compactness of the probability measures $m_x$ in the weak star topology and the continuity of function $S(x,\textbf{j})$. Also for any  $n\in\mathbb{Z}_+$, by (\ref{FundementalFormular}) we have
	\begin{equation}\label{eq:x_0}
	m_{x_1}(\{p\})=\frac1{b^n}\sum_{\textbf{w}\in\varLambda^n}T^nm_{\textbf{w}(x_1)}(\{p\}).
	\end{equation}
	Denote $$p_{\textbf{w}}=\frac{p-S(x_1,\textbf{w})}{\gamma^n},$$ thus $T^nm_{\textbf{w}(x_1)}(\{p\})=m_{\textbf{w}(x_1)}(\{p_{\textbf{w}}\})$.
	Combining this with  (\ref{EqMaxActom}) 
	and (\ref{eq:x_0})
	 we have
	 $$m_{\textbf{w}(x_0)}( \{p_{\textbf{w}}\})=	m_x(\{p\})>0\quad\quad\forall n\in\mathbb{Z}_+,\,\textbf{w}\in\varLambda^n.$$
	  This implies
	 \begin{equation}\label{EqActom}
	 M_0=\sup_{\textbf{w}\in\varLambda^{\#}}|P_{\textbf{w}}|<\infty
	 \end{equation}
	since the supports of the family of probability measures $\{m_x\}_{x\in[0,1]}$ have uniform bound in $\mathbb{R}$.
	
	For any $\textbf{w}\in\varLambda^{\#}$ and $m\in\mathbb{Z}_+$,
	 by the definition of $p_{\textbf{w}}$ we  have
	$$S(x_1,\textbf{w}0_m)-S(x_1,\textbf{w}1_m)=\gamma^{|\textbf{w}|+m}(p_{\textbf{w}1_m}-p_{\textbf{w}0_m})$$
	where $1_m=11\ldots1,0_m=00\ldots0\in\varLambda^m,$ which implies
	\begin{equation}\nonumber
	S(\textbf{w}(x_1),0_m)-S(\textbf{w}(x_1),1_m)=\gamma^m(
	p_{\textbf{w}1_m}-p_{\textbf{w}0_m})
	\end{equation}
	by (\ref{FM}).
	This and  (\ref{EqActom})  implies that
	$$|S(\textbf{w}(x_1),0_m)-S(\textbf{w}(x_1),1_m)|\le2\gamma^m M_0\quad\quad\forall m\in\mathbb{Z}_+.$$ So when m goes to infinity, we have
		$$S(\textbf{w}(x_1),0_{\infty})-S(\textbf{w}(x_1),1_{\infty})=0\quad\quad\forall \textbf{w}\in\varLambda^{\#}.$$
		Therefore 
		$S(x,0_{\infty})-S(x,1_{\infty})\equiv0$ since the set $\{\,\textbf{w}(x_1)\,\}_{\textbf{w}\in\varLambda^{\#}}$ is dense in $[0,1]$, which contradicts condition (H).
\end{proof}

\subsection{Uniform continuity across scales}
Following~\cite{barany2019hausdorff}, we call that a measure $\mu\in\mathscr{P}(\R)$ is {\em uniformly continuous across scales} if for any $\varepsilon>0$ there exists $\delta>0$ satisfied that for each $x\in\mathbb{R}$ and $r\in(0,1)$, we have
\begin{equation}\label{eqn:ucas}
\mu(B(x,\delta r))\le\varepsilon\mu(B(x, r)).
\end{equation}
A family $\mathcal{M}$ of measures in  $\mathscr{P}(\R)$ is called {\em jointly uniformly continuous across scales} if for each $\varepsilon>0$ there exists $\delta>0$ such that (\ref{eqn:ucas}) holds for every $\mu\in \mathcal{M}$, $x\in \R$ and any $r\in (0,1)$. The proof of the following Proposition \ref{prop:uc} is similar to \cite[proposition 4.1]{ren2021dichotomy}. For the readers' convenience we will give the details.
\begin{proposition}\label{prop:uc}
	If the condition (H) holds,
	the family of measures $\{m_x\}_{x\in [0,1]}$ is jointly uniformly continuous across scales.
\end{proposition}
\begin{proof}
	 By Lemma \ref{lemNonatomic},
	for any $\varepsilon>0$ there exists $\delta'=\delta'(\varepsilon)>0$ such that for each $x\in [0,1]$ and  $y\in \mathbb{R}$, we have
	\begin{equation}\label{EqMeasureSmall}
	m_x\,\big({B}(y,\delta')\big)<\varepsilon
	\end{equation}
	since the family of probability measures $m_x$ is compact in the weak star topology. Let  $\delta>0$ be some constant such that the following holds. For any $r\in(0,1)$, there exists $n_r\in\mathbb{N}$ satisfying that
	\begin{equation}\label{Eqn_r1}
	r\delta+2\gamma^{n_r}\parallel\phi\parallel_{\infty}/(1-\gamma)<r
	\end{equation}
	and
	\begin{equation}\label{Eqn_r2}
	\delta r/{\gamma^{n_r}}<\delta'.
	\end{equation}
	Now we shall consider  $\textbf{w}\in\varLambda^{n_r}$ satisfying that
	 $$T^{n_r}m_{\textbf{w}(x)}(\boldsymbol{B}(y,\delta r))>0.$$
	 By the definition of $T^{n_r}m_{\textbf{w}(x)}$ there exists  $\textbf{j}\in\Sigma$ such that $ |S(x,\textbf{w}\textbf{j})-y|\le\delta r$, which implies 
	 $$	|S(x,\textbf{w}\textbf{i})-y|\le\delta r+2\gamma^{n_r}\parallel\phi\parallel_{\infty}/(1-\gamma)\quad\quad\forall\textbf{i}\in\Sigma.$$
	 Combining this with (\ref{Eqn_r1})
	 we have
	 \begin{equation}\label{eq:1}
	 T^{n_r}m_{\textbf{w}(x)}(\boldsymbol{B}(y, r))=1.
	 \end{equation}
	 Also  (\ref{EqMeasureSmall}) and (\ref{Eqn_r2}) implies
	 \begin{equation}\label{eq:number}
	 T^{n_r}m_{\textbf{w}(x)}(\boldsymbol{B}(y,\delta r))=m_{\textbf{w}(x)}\bigg(\boldsymbol{B}\big(\,\frac{y-S(x,\textbf{w})}{\gamma^{n_r}},\,\delta r/\gamma^{n_r}\big)\bigg)<\varepsilon.
	 \end{equation}
	 
	 Finally by (\ref{eq:1}) and (\ref{eq:number}),  we have
	 \begin{equation}\nonumber
	 \begin{aligned}
	 	m_{x}(\boldsymbol{B}(y,\delta r))&=\frac1{b^{n_r}}\sum_{\textbf{w}\in\varLambda^{n_r}}T^{n_r}m_{\textbf{w}(x)}(\boldsymbol{B}(y,\delta r))
	 	\\ &\le\frac{\eps}{b^{n_r}}\#\bigg\{\textbf{w}\in\varLambda^{n_r}:T^nm_{\textbf{w}(x)}(\boldsymbol{B}(y,\delta r))>0\bigg\}
	 	\\&\le \frac{\eps}{b^{n_r}}\sum_{\textbf{w}\in\varLambda^{n_r}}T^{n_r}m_{\textbf{w}(x)}(\boldsymbol{B}(y, r))
	 	\\&=\eps m_{x}(\boldsymbol{B}(y,r)).
	 \end{aligned}
	 \end{equation}
\end{proof}

\subsection{Entropy porosity of $m_x$ }
In this subsection we shall complete the proof of Theorem~\ref{thm:entporous}.
\begin{lemma}\label{lem:boundsinm}
	For any $\varepsilon>0, m\ge M_2(\varepsilon), n\ge N_2(\varepsilon,m)$,
	$$\inf\limits_{x\in [0,1]}\mathbb{\nu}^{n}\left(\left\{\textbf{i}\in \varLambda^{n}: \alpha-\varepsilon<\frac{1}{m}
	H(m_{\textbf{i}(x)}, \mathcal{L}_{m})<\alpha+\varepsilon\right\}\right) >1-\varepsilon.$$
\end{lemma}
\begin{proof}
	We may consider $h_m(x)=\frac{1}{m}
	H(m_{x}, \mathcal{L}_{m})$ and the proof is the same as \cite[Lemma 4.2]{ren2021dichotomy}.
\end{proof}

\begin{lemma}\label{lem:entporous3}
	Under the condition (H), for any $\varepsilon>0$, there exists $\delta>0$ such that if 	
	$m\ge M_3(\varepsilon)$ and $k\ge K_3(\varepsilon,m)$ and if $\left|\frac{1}{k} H(m_x, \mathcal{L}_k)-\alpha\right|<\frac{\delta}{2}$, then $m_x$ is $(\alpha, \varepsilon, m)$-entropy porous from scale $1$ to $k$.
\end{lemma}
\begin{proof}
	Combining Proposition \ref{prop:uc} and Lemma \ref{lem:boundsinm} with (\ref{FundementalFormular}),
	the proof is similar to \cite[Lemma 4.5]{ren2021dichotomy}.

\end{proof}

\begin{proof} [Proof of Theorem~\ref{thm:entporous}]
	Given $\eps>0$, let $\delta$, $M_3(\eps)$ and $K_3(\eps, m)$ be given by Lemma~\ref{lem:entporous3}. For this $\delta>0$, by  Lemma~\ref{lem:boundsinm},  when $k\ge M_2(\delta/2)$ and $n\ge N_2(\delta/2, k)$, 
	$$\nu^{n}\left(\left\{\textbf{i}\in \varLambda^{n}: \left|\frac{1}{k} H(m_{\textbf{i}(0)},\mathcal{L}_k)-\alpha \right|<\frac{\delta}{2}\right\}\right)>1-\delta.$$
	Therefore, when $m\ge M_3(\eps)$, $k\ge \max(K_3(\eps, m), K_2(\delta/2))$ and $n\ge N_2(\delta/2,k)$, the Theorem holds.
\end{proof}

\section{Transversality}\label{tra}
 In this section, under the condition (H) we deduce some quantified estimates for transversality,  which will be used to construct a sequence of partitions $\mathcal{L}_n^{\varLambda^{\#}}$ in Sect. \ref{sec:partitionX}. These partitions  allow us to construct the convolution structures and control the entropy of measures to prove Theorem A. 

 For any $\textbf{i},\,\textbf{j}\in\varLambda^{\#}\cup\Sigma$ and every
 integer $1\le k\le|\textbf{j}|$, let $\textbf{j}_{k}=j_1j_2\ldots j_k$ as usual. 
 We  denote $\textbf{i}<\textbf{j}$  if $\textbf{i}=\textbf{j}_{|\textbf{i}|}$ holds.
 When $|\textbf{j}|<\infty$, let $I_{\textbf{j}}$ be an interval in $[0,1]$ such that $$I_{\textbf{j}}=\bigg[\frac{i_1+i_2b+\dots+i_nb^{n-1}}{b^n},\frac{1+i_1+i_2b+\dots+i_nb^{n-1}}{b^n}\bigg).$$
  The main result of this section is the following Theorem \ref{thmTransversality}.
\begin{theorem}\label{thmTransversality}
	Fix an integer $b\ge2$ and $\gamma\in(0,1)$. Assume $\phi(x)$ is a real analytic $\mathbb{Z}$-periodic function such that $\alpha<\min\{1,\frac{\log b}{\log1/\gamma}\}$ and the condition (H) holds.
	For any $t_0>0$, there exists an integer $t>t_0$, real number $\varDelta_1>0$ and  $\textbf{h},\,\textbf{h}',\,\textbf{a}\in\varLambda^t$ with  the following  property.
	
	For every $z\in I_{\textbf{a}}$ and $\textbf{i},\,\textbf{j}\in \varLambda^{\#}$, if $\textbf{h}<\textbf{i}$, 	$\textbf{h}'<\textbf{j}$, then
	\begin{enumerate}
		\item[(A.1)] $|S'(z,\,\textbf{i})|\,,\,|S'(z,\,\textbf{j})|>\varDelta_1;$
		\item[(A.2)] $|S'(z,\,\textbf{i})-S'(z,\,\textbf{j})|>\varDelta_1.$
		
	\end{enumerate}
\end{theorem}
In fact, (A.2) implies that $\textbf{h}\neq\textbf{h}'$ in Theorem \ref{thmTransversality}.
\begin{lemma}\label{lemNeqZero}
	Under the assumption of Theorem \ref{thmTransversality},
	there exists $x_2\in [0,1]$ and $\textbf{u}\in \Sigma$ such that
	$$S'(x_2,\textbf{u})\neq0.$$
\end{lemma}
\begin{proof} We only need to prove the following claim.
	\begin{claim}\label{claimS1}
		There exists $\textbf{i}\,,\,\textbf{j}\in \Sigma$ and $x_2\in [0,1]$ such that $i_1\neq j_1$ and
		$$S'(x_2,\textbf{i})-S'(x_2,\textbf{j})\neq0.$$
	\end{claim}
	Let us consider the function $f:\mathbb{R}\to\mathbb{R}$ such that
		$$f(x)=\sum_{k=1}^{\infty}\gamma^{k-1}\phi\big(\frac{x}{b^k}\big)\quad\quad\forall x\in\mathbb{R}.$$
		Thus $f(x)=S(x,\textbf{0}_{\infty})$ and $f(x+1)=S(x,\textbf{1}\,\textbf{0}_{\infty})$ for each $x\in[0,1]$ where 
		$\textbf{0}_{\infty}=000\ldots\in\Sigma$ and $\textbf{1}=1\in\varLambda.$
	
	If the claim fails, then we have $f'(x)-f'(x+1)=0$ for each $x\in\mathbb{R}$ since $f'(x)-f'(x+1)$ is a real analytic function in $\mathbb{R}$. Combining this with
	$\sup_{x\in\mathbb{R}}|f(x)|<\infty$ we have $f(x)-f(x+1)=0$ for all $x\in\mathbb{R}$,
	which contracts the condition (H). Thus the claim holds.
\end{proof}
\begin{lemma}\label{lemMinusNeqZero}
	Under the assumption of Theorem \ref{thmTransversality},
	for any $\textbf{v}\,,\textbf{w}\in\varLambda^{\#}$ there exists
	$\textbf{i}\neq\textbf{j}\in\Sigma$ and
	$x_{\textbf{v},\textbf{w}}\in I_{\textbf{v}}$ such that
	$$S'(x_{\textbf{v},\textbf{w}},\textbf{w}\textbf{i})-S'(x_{\textbf{v},\textbf{w}},\textbf{w}\textbf{j})\neq0.$$
\end{lemma}
\begin{proof} If the lemma fails, then there exists  $\textbf{v}\,,\textbf{w}\in\varLambda^{\#}$ such that the following holds.
	For any $\textbf{i},\textbf{j}\in\Sigma$ s.t. $i_1\neq j_1$, we have
	$$S'(x,\textbf{w}\textbf{i})-S'(x,\textbf{w}\textbf{j})=0\quad\forall x\in I_{\textbf{v}}.$$
	Thus, by (\ref{FM}) we have
	$$S'(x,\textbf{i})-S'(x,\textbf{j})=0\quad\forall x\in I_{\textbf{v}\textbf{w}}.$$
	Since $S'(x,\textbf{i})-S'(x,\textbf{j})$ is a real analytic function, we have 
	$$S'(x,\textbf{i})-S'(x,\textbf{j})=0\quad\forall x\in [0,1].$$
	This  contradicts the claim \ref{claimS1} and the Lemma holds.
\end{proof}


\begin{proof}[The proof of Theorem \ref{thmTransversality}]
	By Lemma \ref{lemNeqZero} and $\parallel\phi'\parallel_{\infty}<\infty$, there exists $\varDelta>0$ and $\textbf{y},\textbf{w}\in\varLambda^{\#}$ such that
	\begin{equation}\label{EqSLowBound1}
	|S'(x,\textbf{w}\textbf{t})|>\varDelta
	\end{equation}
	for any $x\in I_{\textbf{y}}$ and every $\textbf{t}\in\varLambda^{\#}\cup\Sigma.$
	Also there exists  $\varDelta'>0$  and $\textbf{y}',\,\textbf{i}\neq\textbf{j}\in\Sigma$ such that the following holds by Lemma \ref{lemMinusNeqZero} and $\parallel\phi'\parallel<\infty$
	\begin{enumerate}
		\item[(1)] $I_{\textbf{y}'}\subset I_{\textbf{y}};$
		\item[(2)] for any  $x\in I_{\textbf{y}'}$ we have $|S'(x,\textbf{w}\textbf{i})-S'(x,\textbf{w}\textbf{j})|\ge\varDelta'.$
	\end{enumerate}
	Choose an integer $t\ge t_0+|\textbf{w}|+|\textbf{y}'|$  such that
	\begin{equation}\label{EqSLowBound0}
	|S'(x,\textbf{w}\textbf{i}_{t-|\textbf{w}|})-S'(x,\textbf{w}\textbf{j}_{t-|\textbf{w}|})|\ge\varDelta'/2\quad \forall\,x\in I_{\textbf{y}'}
	\end{equation}
	and 
	\begin{equation}\label{EqSLowBound2}
	\frac{2(\gamma/b)^{t}\parallel\phi'\parallel_{\infty}}{1-\gamma/b}<\varDelta'/4.
	\end{equation}
	Let $\textbf{y}''\in\varLambda^t$ s.t. $I_{\textbf{y}''}\subset I_{\textbf{y}'}$
	and $\varDelta_1=\min\{\varDelta,\,\varDelta'/4\}$.
	Thus $\textbf{h}=\textbf{w}\textbf{i}_{t-|\textbf{w}|},\,\textbf{h}'=\textbf{w}\textbf{j}_{t-|\textbf{w}|}$ and $\textbf{a}=\textbf{y}''$ are  what we need by (\ref{EqSLowBound1}), (\ref{EqSLowBound0}) and (\ref{EqSLowBound2}). 
\end{proof}

\section{The partitions  of the space $\varLambda^{\#}$}\label{sec:partitionX}
 In this section, we construct a  sequence of partitions $\mathcal{L}_n^{\varLambda^{\#}}$ of the space $\varLambda^{\#}$. The method is similar to \cite[Section 6]{ren2021dichotomy}. Also we  give some useful properties about these partitions with  Theorem \ref{thmTransversality}. 
 
 In the rest of paper, we fix an integer $b\ge2$ and $\gamma\in(0,1).$ Also suppose that 
 $\phi(x)$ is a real analytic $\mathbb{Z}$-periodic function such that the condition (H) holds.
 Combining Theorem \ref{TheoremB} with
Theorem \ref{thmSeperation} and Theorem \ref{thmTransversality},
there exists an integer $t>0$, some constants $\varDelta_1,\, C>0$, a point $x_0\in [0,1)$, set $M\subset\mathbb{Z}_+$ and $\textbf{a},\,\textbf{h},\,\textbf{h}'\in\varLambda^t$  with the following properties.
\begin{enumerate}
	\item[(B.1)] For each  $\textbf{w}\in\varLambda^{t}$, the sequence $(X_n^{\textbf{w},\,x_0}\,)_{n\in\mathbb{Z}_+}$
	is $(\gamma^{C/2}, M)$-exponential separation
	where $X_n^{\textbf{w},\,x_0}$ is from Theorem \ref{thmSeperation};
	\item[(B.2)] $dim(m_{x_0})=\alpha$;
	\item[(B.3)]For every $z\in I_{\textbf{a}}$ and $\textbf{i},\,\textbf{j}\in \varLambda^{\#}$, if $\textbf{h}<\textbf{i}$, 	$\textbf{h}'<\textbf{j}$ then (A.1), (A.2) hold;
		\item[(B.4)] $\# M=\infty.$
\end{enumerate}
 In the rest of paper, we shall fix such elements $\{t,\,x_0,\,C,\, \varDelta_1,\, M,\,\textbf{a},\,\textbf{h},\textbf{h}' \}$.
 Let  $\overline{\pi}:\varLambda^{\#}\to\mathbb{N}\times\mathbb{R}^{3}$ be the map such that
$$\textbf{w}\mapsto\bigg(|\textbf{w}|,\,S(\textbf{w}(x_0),\textbf{h}),\,S(\textbf{w}(x_0),\textbf{h}'),\,S(x_0,\textbf{w})\bigg).$$
\begin{definition} For any integer $n\ge 1$, let 
	$\mathcal{L}_n^{\varLambda^{\#}}$ be the union of all the non-empty subsets of $\varLambda^{\#}$ of the following form
	$$\overline{\pi}^{-1}\left(\{m\}\times I_1 \times I_2  \times J\right),$$
	where $m\in \mathbb{N}, \,I_1,\,I_2\in \mathcal{L}_n,\, J\in \mathcal{L}_{n+[m\log_b 1/\gamma]}.$
	The partition $\mathcal{L}_0^{\varLambda^{\#}}$ consists of non-empty subsets of $\varLambda^{\#}$ of the following form
	$$\overline{\pi}^{-1}\left(\{m\}\times\mathbb{R}\times\mathbb{R}\times J\right),$$
	where $m\in \mathbb{N}, J\in \mathcal{L}_{[m\log_b 1/\gamma]}.$
\end{definition}
Note that $\textbf{w}(x_0)$ is defined by (\ref{eq:text}).
\begin{lemma}\label{lem:constantR}
There exists $R>0$ such that the following holds.

Let $n,\,i$ be positive integers such that $\hat{n}>t$.  For each $\textbf{u}, \textbf{v}\in \varLambda^{\hat{n}-t}$, if $\textbf{u}\textbf{a}$ and $\textbf{v}\textbf{a}$ belong to the same element of $\mathcal{L}_i^{\varLambda^{\#}}$, then for  any $\textbf{q}\in\varLambda^{\hat{i}}$  and $\textbf{j},\,\textbf{i}\in \Sigma$, we have
  \begin{equation}\label{EqSupport}
  |S(x_0,\textbf{u}\textbf{a}\textbf{q}\,\textbf{j})-S(x_0,\textbf{v}\textbf{a}\textbf{q}\,\textbf{i})|\le R b^{-(n+i)}.
  \end{equation}
\end{lemma}
\begin{proof} We only need to consider the case $\textbf{u}\neq\textbf{v}$.
	By the definition of the partition  $\mathcal{L}^{\varLambda^{\#}}_i$, we have
	$$|S(\textbf{u}\textbf{a}(x_0),\textbf{h})-S(\textbf{v}\textbf{a}(x_0),\textbf{h})|\le1/b^i.$$
	Combining this with  (A.1) we have
	$$|\textbf{u}\textbf{a}(x_0)-\textbf{v}\textbf{a}(x_0)|=O(1/b^i)$$
	since $\textbf{u}\textbf{a}(x_0),\,\textbf{v}\textbf{a}(x_0)\in I_{\textbf{a}}$
	( $I_{\textbf{a}}$ is defined in Sect. \ref{tra}).
	 Thus
	\begin{equation}\label{EqSpart2}
	|S(\textbf{u}\textbf{a}(x_0),\textbf{q})-S(\textbf{v}\textbf{a}(x_0),\textbf{q})|=O(1/b^i).
	\end{equation}
	Also we have
	\begin{equation}\label{EqSpart1}
	|S(x_0,\textbf{u}\textbf{a})-S(x_0,\textbf{v}\textbf{a})|=O(1/b^{n+i})
	\end{equation}
	by the definition of the  partition  $\mathcal{L}^{\varLambda^{\#}}_i$. Combining (\ref{EqSpart1}), (\ref{EqSpart2}) with 
	$$S(x_0,\textbf{u}\textbf{a}\textbf{q}\,\textbf{j})=S(x_0,\textbf{u}\textbf{a})+\gamma^{\hat{n}}S(\textbf{u}\textbf{a}(x_0),\textbf{q})+\gamma^{\hat{n}+\hat{i}}S(\textbf{u}\textbf{a}\textbf{q}(x_0),\textbf{j})$$
	and 
	$$S(x_0,\textbf{v}\textbf{a}\textbf{q}\,\textbf{i})=S(x_0,\textbf{v}\textbf{a})+\gamma^{\hat{n}}S(\textbf{v}\textbf{a}(x_0),\textbf{q})+\gamma^{\hat{n}+\hat{i}}S(\textbf{v}\textbf{a}\textbf{q}(x_0),\textbf{i}),$$
	we have (\ref{EqSupport}) holds for some constant $R>0$ large enough.
	\end{proof}

	It is necessary to introduce the following notation. For any probability measure $\xi\in\mathscr{P}(\varLambda^{\#})$ and $\textbf{u}\in\varLambda^{\#}$, let probability measure  $A_{\textbf{u}}(\xi)\in\mathscr{P}(\mathbb{R})$ be such that
	\begin{equation}\label{measure}
	A_{\textbf{u}}(\xi)=\sum_{\textbf{w}\in\text{supp}(\xi)}\xi(\{\textbf{w}\})\,\delta_{S(x_0,\textbf{w}\,\textbf{u})}.
	\end{equation}
	
	
\begin{lemma}\label{lem:KeyPointLow}
	There exists a constant $L_1>0$ such that the following holds for positive  integers $n,\,k,\,i$
	 and $\textbf{q}\in\varLambda^{\#}$.
	If $\xi$ is a probability measure supported in an element of $\mathcal{L}_i^{\varLambda^{\#}}$ such that 
	$|\textbf{w}|=\hat{n}$ and $I_{\textbf{w}}\subset I_{\textbf{a}}$ for
	each element $\textbf{w}$ in the support of $\xi$, then
	$$H\big(\xi,\mathcal{L}_{i+k}^{\varLambda^{\#}}\big)
	\le H\big(A_{\textbf{h}'\textbf{q}}(\xi),\mathcal{L}_{i+k+n}\big)+H\big( A_{\textbf{h}\textbf{q}}(\xi),\mathcal{L}_{i+k+n}\big)+L_1.$$
\end{lemma}
\begin{proof}
	Define $F:\text{supp}(\xi)\to\mathbb{R}^2$, by
	$$\textbf{w}\to\big(S(x_0,\textbf{w}\textbf{h}\textbf{q}),\,S(x_0,\textbf{w}\textbf{h}'\textbf{q})\big).$$
	\begin{claim}\label{cla:entropy}
		There exists a constant $L_1>0$ such that
		$$H\big(\xi,\mathcal{L}_{i+k}^{\varLambda^{\#}}\big)
		\le H\big( F\xi,\mathcal{L}_{i+k+n}^{\mathbb{R}^2}\big)+L_1.$$
	\end{claim}
	To prove this claim, take $I\in\mathcal{L}_{i+k+n}^{\mathbb{R}^2}$.
	It suffices to show that the cardinality of the set $\{J\in\mathcal{L}_{i+k}^{\varLambda^{\#}}\,\big|\, J\cap F^{-1}(I)\ne\emptyset\mbox{ and }J\cap \text{supp}(\xi)\ne\emptyset\}$ is uniformly bounded. For any $\textbf{w}^{(m)}\in\text{supp}(\xi),\,\text{with}\,F(\textbf{w}^{(m)})\in I,\,m=1,2,$ we have
	$$\bigg|\bigg(S(x_0,\textbf{w}^{(1)}\textbf{h}\textbf{q})-S(x_0,\textbf{w}^{(1)}\textbf{h}'\textbf{q})\bigg)-\bigg(S(x_0,\textbf{w}^{(2)}\textbf{h}\textbf{q})-S(x_0,\textbf{w}^{(2)}\textbf{h}'\textbf{q})\bigg)\bigg|=O(b^{-(i+k+n)})$$
	by the definition of the function $F$,
	which implies
		$$\bigg|\bigg(S(\textbf{w}^{(1)}(x_0),\textbf{h}\textbf{q})-S(\textbf{w}^{(1)}(x_0),\textbf{h}'\textbf{q})\bigg)-\bigg(S(\textbf{w}^{(2)}(x_0),\textbf{h}\textbf{q})-S(\textbf{w}^{(2)}(x_0),\textbf{h}'\textbf{q})\bigg)\bigg|=O(b^{-(k+i)})$$
		by (\ref{FM}).
		Therefore
		$$|\textbf{w}^{(1)}(x_0)-\textbf{w}^{(2)}(x_0)|=O(b^{-(k+i)})$$
		by (A.2)  and  $\textbf{w}^{(1)}(x_0),\,\textbf{w}^{(2)}(x_0)\in I_{\textbf{a}}$
		( (A.2) is in Sect. \ref{tra}).
		So we have
		\begin{equation}\label{EqProjection1}
		\big|S(\textbf{w}^{(1)}(x_0),\textbf{h})-S(\textbf{w}^{(2)}(x_0),\textbf{h})\big|=O(b^{-(k+i)})
		\end{equation}
		\begin{equation}\label{EqProjection2}
		\big|S(\textbf{w}^{(1)}(x_0),\textbf{h}')-S(\textbf{w}^{(2)}(x_0),\textbf{h}')\big|=O(b^{-(k+i)})
		\end{equation}
		and
		$$\big|S(\textbf{w}^{(1)}(x_0),\textbf{h}\textbf{q})-S(\textbf{w}^{(2)}(x_0),\textbf{h}\textbf{q})\big|=O(b^{-(k+i)}).$$
	Combining this with
	$$\big|S(x_0,\textbf{w}^{(1)}\textbf{h}\textbf{q})-S(x_0,\textbf{w}^{(2)}\textbf{h}\textbf{q})\big|=O(b^{-(k+i+n)})$$
	i.e.
	$$\big|S(x_0,\textbf{w}^{(1)})-S(x_0,\textbf{w}^{(2)})-\gamma^{\hat{n}}\big(S(\textbf{w}^{(1)}(x_0),\textbf{h}\textbf{q})-S(\textbf{w}^{(2)}(x_0),\textbf{h}\textbf{q})\big)\big|=O(b^{-(k+i+n)})$$
	we have
	\begin{equation}\label{EqProjection3}
	\big|S(x_0,\textbf{w}^{(1)})-S(x_0,\textbf{w}^{(2)})\big|=O(b^{-(k+i+n)}).
	\end{equation}
	 This with (\ref{EqProjection1}) and (\ref{EqProjection2}) imply that our claim holds by the definition of $\mathcal{L}_{i+k}^{\varLambda^{\#}}.$
	 
	 Finally let us consider the functions $F_1,\,F_2:\text{supp}(\xi)\to\mathbb{R},$ by
	$$F_1(\textbf{w})=S(x_0,\textbf{w}\textbf{h}\textbf{q})\quad\quad F_2(\textbf{w})=S(x_0,\textbf{w}\textbf{h}'\textbf{q}).$$
	 Then
	 $$H\big(\ F\xi,\mathcal{L}_{i+k+n}^{\mathbb{R}^2}\big)=H\big(\ \xi,F^{-1}(\mathcal{L}_{i+k+n}^{\mathbb{R}^2})\big)=H\big(\ \xi,\bigvee _{j=1}^{2}F_{j}^{-1}(\mathcal{L}_{i+k+n})\big)$$
	 $$
	 \le \sum\limits_{j=1}^{2}H\big(\ \xi,F_{j}^{-1}(\mathcal{L}_{i+k+n})\big)=
	 H\big( A_{\textbf{h}'\textbf{q}}(\xi),\mathcal{L}_{i+k+n}\big)+H\big(\ A_{\textbf{h}\textbf{q}}(\xi),\mathcal{L}_{i+k+n}\big).$$
	 Combining this with Claim \ref{cla:entropy} this lemma holds.
\end{proof}

\section{Proof of Theorem A}\label{sec:pfThmA}
In this section, we will apply Hochman's criterion on entropy increasing to complete the proof of Theorem A.  For a  discrete measure $\xi\in\mathscr{P}(\varLambda^{\#})$ and $\textbf{q}\in\varLambda^{\#}$, let $B_{\textbf{q}}(\xi)$ denote the  Borel probability  measures in $\mathbb{R}$ such that for any Borel subset $A$ of  $\mathbb{R}$, 
\begin{equation}\label{measure2}
B_{\textbf{q}}(\xi)(A)=\xi\times\nu^{\mathbb{Z}_+}\big(\,\big\{\,(\textbf{w},\, \textbf{j})\in\varLambda^{\#}\times \Sigma: S(x_0,\textbf{w}\,\textbf{q}\,\textbf{j})\in A\big\}\,\big).
\end{equation}

  The key point is to introduce the discrete measure in $\varLambda^{\#}$
  \begin{equation}\label{measure3}
  \theta^{\,\textbf{u}}_n:=\frac1{b^{\hat{n}-t}}\sum_{\textbf{w}\in\varLambda^{\hat{n}-t}}\delta_{\textbf{w}\textbf{u}}
  \end{equation}
   for any $\textbf{u}\in\varLambda^t$ and integer $n$ such that $\hat{n}>t$.
    Let $n,\,i$ be integers satisfying that $\hat{n},\,\hat{i}>t$.
    (\ref{FundementalFormular}) implies that, for any Borel subset  $A$ of  $\mathbb{R}$, $$m_{x_0}(A)=\frac1{b^{2t}}\sum_{(\textbf{u},\textbf{v})\in\varLambda^t\times\varLambda^t}\frac1{b^{\hat{i}-t}}\sum_{\textbf{q}\in\varLambda^{\hat{i}-t }}
   \frac1{b^{\hat{n}-t}}\sum_{\textbf{w}\in\varLambda^{\hat{n}-t}}\nu^{\mathbb{Z}_+}\bigg(\, \textbf{j}\in\Sigma: S\big(x_0,\textbf{w}\textbf{u}\textbf{v}\textbf{q}\,\textbf{j}\big)\in A\}\bigg).$$
   Therefore we have
   $$m_{x_0}(A)=\frac1{b^{2t}}\sum_{(\textbf{u},\textbf{v})\in\varLambda^t\times\varLambda^t}\frac1{b^{\hat{i}-t}}\sum_{\textbf{q}\in\varLambda^{\hat{i}-t }}\theta^{\,\textbf{u}}_n\times\nu^{\mathbb{Z}_+}\bigg(\,\bigg\{ (\textbf{i},\,\textbf{j})\in\varLambda^{\#}\times\Sigma: S(x_0,\textbf{i}\textbf{v}\textbf{q}\,\textbf{j})\in A\bigg\}\bigg)
  $$ 
   by the definition of $\theta^{\,\textbf{u}}_n.$ Thus we have
  \begin{equation}\label{EqFormular}
  m_{x_0}=\frac1{b^{2t}}\sum_{(\textbf{u},\textbf{v})\in\varLambda^t\times\varLambda^t}\frac1{b^{\hat{i}-t}}\sum_{\textbf{q}\in\varLambda^{\hat{i}-t }} B_{\textbf{v}\textbf{q}}
   (\theta^{\textbf{u}}_n)
  \end{equation}
  by the definition of $B_{\textbf{v}\textbf{q}} (\theta^{\textbf{u}}_n).$
 
 \subsection{The entropy of $\theta^{\,\textbf{a}}_n$}\label{FunctionEntropy}
 We start with analyzing the entropy of $\theta^{\,\textbf{a}}_n$ with respect to the partitions $\mathcal{L}_i^{\varLambda^{\#}}$.
 \begin{lemma}\label{lem:thetanL0} $$\lim_{n\to\infty}\frac{1}{n}{H}\left(\theta^{\,\textbf{a}}_n,\mathcal{L}_0^{\varLambda^{\#}}\right)=\lim_{n\to\infty} \frac{1}{n} H(m_{x_0}, \mathcal{L}_n)=\alpha.$$
 \end{lemma}
\begin{proof}
	For $\hat{n}>t$,
	define $\pi_n\,:\Sigma\to\,\varLambda^{\hat{n}}$ by $\pi(\textbf{j})=\textbf{j}_{\hat{n}-t}\textbf{a}$ and
	$f_n\,:\varLambda^{\hat{n}}\to \,\mathbb{R}$ by 
	$f_n(\textbf{w})=S(x_0,\textbf{w})$
	where $\textbf{j}_{\hat{n}-t}=j_1j_2\ldots j_{\hat{n}-t}$.
	 Recall the map $S_{x_0}:\Sigma\to\,\mathbb{R}\,;\,\textbf{j}\,\mapsto\,S(x_0,\textbf{j}),$ so we have $f_n\pi_n-S_{x_0}=O(b^{-n}).$ Thus
	$$H(m_{x_0}, \mathcal{L}_n)=H\left(S_{x_0}\nu^{\mathbb{Z}_+}, \mathcal{L}_n \right)=H(f_n\pi_n\nu^{\mathbb{Z}_+}, \mathcal{L}_n )+O(1)=H(f_n\theta^{\,\textbf{a}}_n, \mathcal{L}_n )+O(1)$$
	since $\pi_n\nu^{\mathbb{Z}_+}=\theta^{\,\textbf{a}}_n$. 
	By $\lim_{n\to\infty} \frac{1}{n} H(m_{x_0}, \mathcal{L}_n)=\alpha$, we have
	$$\lim_{n\to\infty}\frac{1}{n}{H}\left(f_n\theta^{\,\textbf{a}}_n, \mathcal{L}_n \right)=\alpha.$$
	Since ${H}\left(f_n\theta^{\,\textbf{a}}_n, \mathcal{L}_n \right)={H}\left(\theta^{\,\textbf{a}}_n,\mathcal{L}_0^{\varLambda^{\#}}\right)$, the lemma holds.
\end{proof}

 \begin{lemma}\label{lem:thetanLcn} 
 	 $$\lim_{n\to\infty,\,n\in M}\frac{1}{n}{H}\left(\theta^{\,\textbf{a}}_n,\mathcal{L}_{Cn}^{\varLambda^{\#}}\right)=\frac {\log b}{\log(1/\gamma)}$$
 	 where $M$ is from Sect. \ref{sec:partitionX}.	
 \end{lemma}
\begin{proof}
	For $n\in M$ large enough, we have
	$$|S(x_0,\textbf{j} \textbf{a} )-S(x_0,\textbf{i} \textbf{a} )|>b^{-cn}\quad\quad\forall\textbf{j}\neq\textbf{i}\in\varLambda^{\hat{n}-t}$$
	by (B.1).
	Thus $${H}\left(\theta^{\,\textbf{a}}_n,\mathcal{L}_{Cn}^{\varLambda^{\#}}\right)=\hat{n}-t$$
	by the definition of $\mathcal{L}_{Cn}^{\varLambda^{\#}}$.
	Since $\lim_{n\to\infty} \hat{n}/n=\log_{1/\gamma}b$, the lemma holds.
\end{proof}

	\subsection{Decomposition of entropy}
	In the following Lemma \ref{lem:thetadec}, we decompose the entropy of $\theta^{\,\textbf{a}}_n$ and $m_{x_0}$ into small scales.  
	
	\begin{lemma}\label{lem:thetadec}
		For any $\varepsilon>0$, there exists $C_0(\varepsilon)>0$ such that the following holds.
		 For any positive integers $k,n$ satisfying that $n> C_0(\varepsilon)k$, we have
		\begin{equation}\label{eqn:thetadec}
		\frac{1}{Cn}H(\theta_n^{\,\textbf{a}}, \mathcal{L}_{Cn}^{\varLambda^{\#}}|\mathcal{L}_0^{\varLambda^{\#}})\le \mathbb{E}_{0\le i<Cn}^{\theta_n^{\,\textbf{a}}} \left[\frac{1}{k}H((\theta_n^{\,\textbf{a}})_{\textbf{w},i}, \mathcal{L}_{i+k}^{\varLambda^{\#}})\right]+\varepsilon,
		\end{equation}
		\begin{equation}\label{eqn:thetamudec}
		\frac{1}{Cn} H(m_{x_0}, \mathcal{L}_{(C+1)n}|\mathcal{L}_{n})\ge \frac1{b^{2t}}\sum_{(\textbf{u},\textbf{v})\in\varLambda^t\times\varLambda^t}\frac1{Cn}\sum_{0\le i<Cn,\, t<\hat{i} }\frac1{b^{\hat{i}-t}}\sum_{\textbf{q}\in\varLambda^{\hat{i}-t}}\bigg[\frac1k H(B_{\textbf{v}\textbf{q}}(\theta^{\,\textbf{u}}_n),\,\mathcal{L}_{i+k+n}|\mathcal{L}_{i+n})\bigg]-\varepsilon.
		\end{equation}
	\end{lemma}
    \begin{proof}
    	The proof  is similar to \cite[Lemma7.3]{ren2021dichotomy}, thus we only give sketchy  proof for (\ref{eqn:thetamudec}).
    	By the same method as \cite[Lemma7.3]{ren2021dichotomy},  when $n/k$ is large enough, we have
    	$$\frac{1}{Cn} H(m_{x_0}, \mathcal{L}_{(C+1)n}|\mathcal{L}_{n})\ge \frac{1}{Cn}\sum_{0\le i< Cn}\left[\frac{1}{k} H(m_{x_0}, \mathcal{L}_{i+k+n}|\mathcal{L}_{i+n})\right]-\eps.$$
        Combining this with (\ref{EqFormular}),
         thus (\ref{eqn:thetamudec}) holds by concavity of conditional entropy.
    	   	
    \end{proof}

\begin{corollary}\label{p}
	If $\alpha<\frac{\log b}{\log1/\gamma}$, then
	there exists $\varDelta_2=\varDelta_2(\alpha)>0$, $p_0=p_0(\alpha,\,\varDelta_2)\in(0,1)$ and $C_1=C_1(\alpha,\,\varDelta_2),\,N_0=N_0(\alpha,\,\varDelta_2)>0$ such that the following holds.
	For any $k\in\mathbb{Z}_+$, $n\in M$  satisfied that $n/k> C_1,\,n\ge N_0$,
     we have
	$$\mathbb{P}_{0\le i< Cn}^{\theta_n^{\,\textbf{a}}} \left[\frac{1}{k}H((\theta_n^{\,\textbf{a}})_{\textbf{w},i}, \mathcal{L}_{i+k}^{\varLambda^{\#}})>\varDelta_2\right]>p_0. $$
\end{corollary}	
\begin{proof}
	Let $\varDelta_3=\frac14(\frac {\log b}{\log(1/\gamma)}-\alpha)>0.$
	Combining (\ref{eqn:thetadec}) with Lemma \ref{lem:thetanL0} and Lemma \ref{lem:thetanLcn} we have
	\begin{equation}\label{EqELowBound}
	\mathbb{E}_{0\le i<Cn}^{\theta_n^{\,\textbf{a}}} \left[\frac{1}{k}H((\theta_n^{\textbf{a}})_{\textbf{w},i}, \mathcal{L}_{i+k}^{\varLambda^{\#}})\right]\ge \varDelta_3
	\end{equation}
	when  $n> C_0(\varDelta_3)k$ and $n\in M$ is large enough.
	Also  there is constant $L_2=L_2(\phi,\gamma)>0$ such that
	each element of $\mathcal{L}_{i}^{\varLambda^{\#}}$ contains at most $b^{L_2k}$ elements of $\mathcal{L}_{i+k}^{\varLambda^{\#}}$ (See \cite[Lemma 6.1]{ren2021dichotomy} for the idea of the proof).	 Therefore for any $\textbf{w}\in\text{supp}(\theta_n^{\,\textbf{a}})$, we have
	$$\frac{1}{k}H((\theta_n^{\,\textbf{a}})_{\textbf{w},i}, \mathcal{L}_{i+k}^{\varLambda^{\#}})\le L_2\quad\quad\forall\,0\le i< Cn.$$
	 Combining this with (\ref{EqELowBound}) our claim holds.
\end{proof}

\subsection{The lower bound estimate of the conditon entropy of $B_{\textbf{q}}(\theta^{\,\textbf{u}}_n)$}
The following Lemma \ref{lemTrivialLowBound} is a trivial lower bound estimate of condition entropy by concavity of conditional entropy. Recall $\textbf{q}(0)=\frac{q_1+q_2b+\dots+q_{\hat{i}}\,b^{\hat{i}-1}}{b^{\hat{i}}}$ for each $i\in\mathbb{Z}_+$ and $\textbf{q}\in\varLambda^{\hat{i}}$. The following $B_{\textbf{q}}(\theta^{\,\textbf{u}}_n)$ is defined by (\ref{measure2}).
 \begin{lemma}\label{lemTrivialLowBound}
 	For any $\eps>0$,
 	if $k,\,n,\,i$ are positive integers with $\hat{n}>t,i> C_2k$ and $k\ge K_4(\varepsilon)$, then for any $\textbf{u}\in\varLambda^t$ and $\textbf{q}\in\varLambda^{\hat{i}}$, we have
 	$$\frac1k H(B_{\textbf{q}}(\theta^{\,\textbf{u}}_n),\,\mathcal{L}_{i+k+n}|\mathcal{L}_{i+n})\ge\frac1k H(m_{\textbf{q}(0)},\,\mathcal{L}_{k})-\varepsilon$$
 	where $C_2=2\log\gamma^{-1}/\log b$.
 \end{lemma}
\begin{proof}
	By the concavity of conditional entropy and the definition of $\theta^{\,\textbf{u}}_n$, 
	\begin{equation}\label{EqTrivialLow}
	\frac1k H(B_{\textbf{q}}(\theta^{\,\textbf{u}}_n),\,\mathcal{L}_{i+k+n}|\mathcal{L}_{i+n})\ge\frac1{b^{\hat{n}-t}}\sum_{\textbf{w}\in\varLambda^{\hat{n}-t}}\frac1k H(B_{\textbf{q}}(\delta_{\textbf{w}\textbf{u}}),\,\mathcal{L}_{i+k+n}|\mathcal{L}_{i+n}).
	\end{equation}
	Since the measure $B_{\textbf{q}}(\delta_{\textbf{w}\textbf{u}})$ is support in an interval of length $O(b^{-(i+n)})$ by (\ref{F}) and the definition of $B_{\textbf{q}}(\delta_{\textbf{w}\textbf{u}})$, we have
	\begin{equation}\label{EqTrivalO1}
	H(B_{\textbf{q}}(\delta_{\textbf{w}\textbf{u}}),\,\mathcal{L}_{i+n})=O(1).
	\end{equation}
	Also for any $\textbf{j}\in\Sigma$, we have
	$$S(x_0,\textbf{w}\textbf{u}\textbf{q}\,\textbf{j})-\big(S(x_0,\textbf{w}\textbf{u}\textbf{q})+\gamma^{\hat{n}+\hat{i}}S(\textbf{q}(0),\textbf{j})\big)=\gamma^{\hat{n}+\hat{i}}\big(S(\textbf{w}\textbf{u}\textbf{q}(x_0)\,\textbf{j})-S(\textbf{q}(0),\textbf{j})\big).$$
	Combining this with 
	$$|\textbf{w}\textbf{u}\textbf{q}(x_0)-\textbf{q}(0)|=O(b^{-\hat{i}})$$
	and $\hat{i}\ge \frac{\log b}{\log1/\gamma}i>\frac{\log b}{\log1/\gamma} C_2k>k$
	we have
	$$|	S(x_0,\textbf{w}\textbf{u}\textbf{q}\,\textbf{j})-f\circ S(\textbf{q}(0),\textbf{j}\big)|=O(b^{-(n+i+k)})$$
	where
	$f(s)=S(x_0,\textbf{w}\textbf{u}\textbf{q})+\gamma^{\hat{n}+\hat{i}}s,\,\forall s\in\mathbb{R}.$ Thus the following holds
	$$	H(B_{\textbf{q}}(\delta_{\textbf{w}\textbf{u}}),\,\mathcal{L}_{i+n+k})=	H(fm_{\textbf{q}(0)},\,\mathcal{L}_{i+n+k})+O(1)=H(m_{\textbf{q}(0)},\,\mathcal{L}_{k})+O(1)$$
	Combining this with (\ref{EqTrivalO1})
	we have
	$$H(B_{\textbf{q}}(\delta_{\textbf{w}\textbf{u}}),\,\mathcal{L}_{i+k+n}|\mathcal{L}_{i+n})=H(m_{\textbf{q}(0)},\,\mathcal{L}_{k})+O(1).$$
	Let  $k$ be large enough, thus the lemma holds by (\ref{EqTrivialLow}).
\end{proof}
	We shall  use the Hochman's criterion on entropy increase (See Theorem \ref{thm:hochmanentgrow}) to give the nontrivial lower bound  estimate of the entropy of $B_{\textbf{q}}(\theta^{\,\textbf{a}}_n)$ when $\textbf{h}<\textbf{q}$ or  $\textbf{h}'<\textbf{q}$ in Sect. \ref{proveA}. Thus we need the following result. For any measure $\mathbb{P}\in\mathscr{P}(\mathbb{R})$ and $r>0$, let 
	$$r\mathbb{P}:=h_r\mathbb{P}$$
	where $h_r(s)=rs,\,\forall s\in\mathbb{R}.$
 \begin{lemma}\label{lemlLowBound}
	For any $\varepsilon>0$, there exists $K_5(\eps)>0$ such that  when $\hat{n}>t,\,i> C_2k$ and $k\ge K_5(\varepsilon)$, the following holds.
	For any $\textbf{q}\in\varLambda^{\hat{i}}$ if $\textbf{h}<\textbf{q}$ or $\textbf{h}'<\textbf{q}$ , then
	$$\frac1k H(\,B_{\textbf{q}}(\theta^{\,\textbf{a}}_n),\,\mathcal{L}_{i+k+n}|\mathcal{L}_{i+n})\ge\mathbb{E}^{\theta^{\,\textbf{a}}_n}_i\bigg[\frac1k H(\bigg\{\gamma^{-(\hat{i}+\hat{n})}A_{\textbf{q}}\bigg((\theta^{\,\textbf{a}}_n)_{\textbf{w},i}\bigg)\bigg\}* m_{\textbf{q}(0)},\,\mathcal{L}_{k})\bigg]-\varepsilon$$
	where $C_2$ is from Lemma \ref{lemTrivialLowBound}.
\end{lemma}	
Recall that $A_{\textbf{q}}\big((\theta^{\,\textbf{a}}_n)_{\textbf{w},i}\big)$ is defined by (\ref{measure}) and $B_{\textbf{q}}(\theta^{\,\textbf{a}}_n)$ is defined by (\ref{measure2}).
\begin{proof}
	By concavity of conditional entropy we have
	\begin{equation}\label{EqEntropyLow}
	\frac1k H(B_{\textbf{q}}(\theta^{\,\textbf{a}}_n),\,\mathcal{L}_{i+k+n}|\mathcal{L}_{i+n})\ge\mathbb{E}^{\theta^{\,\textbf{a}}_n}_i\bigg[\frac1k H\bigg(B_{\textbf{q}}\big((\theta^{\textbf{\,a}}_n)_{\textbf{w},i }\big),\,\mathcal{L}_{i+k+n}|\mathcal{L}_{i+n}\bigg)\bigg].
	\end{equation}
	For each $\textbf{w}\in\text{supp}(\theta^{\,\textbf{a}}_n)$, let $\varPsi=(\theta^{\,\textbf{a}}_n)_{\textbf{w},i}$.
	\begin{claim}\label{claim1}
		\begin{equation}\label{Eqright}
		H(B_{\textbf{q}}(\varPsi),\,\mathcal{L}_{i+n})=O(1).
		\end{equation}
	\end{claim}
	For each $\textbf{w}',\,\textbf{w}''\in\text{supp}(\varPsi)$ and $\textbf{j}',\,\textbf{j}''\in\Sigma,$ we have
	$$|S(x,\textbf{w}'\,\textbf{q}\,\textbf{j}')-S(x,\textbf{w}''\,\textbf{q}\,\textbf{j}'')|\le R b^{-(n+i)}$$
	by Lemma \ref{lem:constantR}, which implies $B_{\textbf{q}}(\varPsi) $ is supported in an interval of length $O(b^{-(i+n)}).$
	Thus the claim holds.
	\begin{claim}\label{cla:convolution}
		$$	H(B_{\textbf{q}}(\varPsi),\,\mathcal{L}_{i+n+k})=H\big(A_{\textbf{q}}(\varPsi)*[\gamma^{\hat{n}+\hat{i}} m_{\textbf{q}(0)}],\,\mathcal{L}_{i+n+k}\big)+O(1).$$
	\end{claim}
    To prove the claim let us consider the map 
    $$F\,:\text{supp}(\varPsi)\times\Sigma\,\to\,\mathbb{R}\quad\quad(\textbf{w}',\,\textbf{j})\,\to\,S(x_0,\textbf{w}'\textbf{q}\,\textbf{j})$$
    and
    the map 
    $$Q\,:\text{supp}(\varPsi)\times\Sigma\,\to\,\mathbb{R}\quad\quad(\textbf{w}',\,\textbf{j})\,\to\,S(x_0,\textbf{w}'\textbf{q})+\gamma^{\hat{i}+\hat{n}}S(\textbf{q}(0),\textbf{j}).$$
    Since $|\textbf{q}(0)-\textbf{w}'\textbf{q}(x_0)|=O(b^{-\hat{i}})$ and $\hat{i}>k$, we have
    $$|F(\textbf{w}',\,\textbf{j})-Q(\textbf{w}',\,\textbf{j})|=O(b^{-(n+i+k)}),$$
    which implies that
    $$H(F(\varPsi\times\nu^{\mathbb{Z}_+}),\,\mathcal{L}_{i+n+k})=H(Q(\varPsi\times\nu^{\mathbb{Z}_+}),\,\mathcal{L}_{i+n+k})+O(1).$$
    Combining this with
    $$H(F(\varPsi\times\nu^{\mathbb{Z}_+}),\,\mathcal{L}_{i+n+k})=H(B_{\textbf{q}}(\varPsi),\,\mathcal{L}_{i+n+k})$$
    and
    $$H(Q(\varPsi\times\nu^{\mathbb{Z}_+}),\,\mathcal{L}_{i+n+k})=H(A_{\textbf{q}}(\varPsi)*[\gamma^{\hat{n}+\hat{i}} m_{\textbf{q}(0)}],\,\mathcal{L}_{i+n+k}),$$
   the claim holds.
   
    Finally  Claim \ref{claim1} and Claim \ref{cla:convolution}  imply that
    $$	H(B_{\textbf{q}}(\varPsi),\,\mathcal{L}_{i+n+k}|\mathcal{L}_{i+n})=H([\gamma^{-(\hat{n}+\hat{i})}A_{\textbf{q}}(\varPsi)]* m_{\textbf{q}(0)},\,\mathcal{L}_{k})+O(1).$$
    Combining this with (\ref{EqEntropyLow}) the lemma holds when k is large enough.
\end{proof}

\subsection{The proof of Theorem A}\label{proveA}
The following Theorem \ref{thm:hochmanentgrow} is a version of Hochman's entropy increasing criterion, see  ~\cite[Theorem 2.8]{hochman2014self} and~\cite[Theorem 4.1]{barany2019hausdorff}.
\begin{theorem}[Hochman]\label{thm:hochmanentgrow}
	For any $\varepsilon>0$ and $m\in\mathbb{Z}_+$ there exists $\delta=\delta(\varepsilon,m)>0$ such that for
	$k>K(\varepsilon,\delta,m)$, $n\in\mathbb{N}$, and $\tau,\theta\in\boldsymbol{\mathscr{P}}(\mathbb{R})$,
	if
	\begin{enumerate}
		\item[(1)] $\text{diam}(\text{supp}(\tau)),\text{diam}(\text{supp}(\theta))\le b^{-n}$,
		\item[(2)] $\tau$ is $(1-\varepsilon,\frac{\varepsilon}2,m)$-entropy porous from scales $n$ to $n+k$,
		\item [(3)]$\frac{1}{k}H(\theta,\mathcal{L}_{n+k})>\varepsilon$,
	\end{enumerate}	
	then
	$$\frac{1}{k}H(\theta\ast\tau,\mathcal{L}_{n+k})\ge\frac{1}{k}H(\tau,\mathcal{L}_{n+k})+\delta.$$ 
\end{theorem}
 We have the following Lemma.
\begin{lemma}\label{lem:EntropyIncrease}
If $\alpha<\{1,\,\frac{\log b}{\log1/\gamma}\}$, then
there exists a constant $\varDelta_4=\varDelta_4(\alpha)>0$ such that the following holds.

 For any $\eps>0$ and positive integers $n,\,i,\,k$ when $\hat{n}>t,$  $k\ge K_6=K_6(\eps,\varDelta_4,\alpha)$ and $i\ge I_1=I_1(\eps,\varDelta_4,\alpha,k)$, we have
\begin{equation}\label{EqEntropyIncrease}
\begin{aligned}
&\frac1{b^{\hat{i}-t}}\sum_{\textbf{q}\in\varLambda^{\hat{i}-t}}\bigg[\frac1k H(B_{\textbf{h}\textbf{q}}(\theta^{\,\textbf{a}}_n),\,\mathcal{L}_{i+k+n}|\mathcal{L}_{i+n})+\frac1k H(B_{\textbf{h}'\textbf{q}}(\theta^{\,\textbf{a}}_n),\,\mathcal{L}_{i+k+n}|\mathcal{L}_{i+n})\bigg]\\
\\&\ge \varDelta_4 P_{i,k}+ \frac1{b^{\hat{i}-t}}\sum_{\textbf{q}\in\varLambda^{\hat{i}-t}}\bigg[\frac1k H(m_{\textbf{h}\textbf{q}(0)},\,\mathcal{L}_{k})+\frac1k H(m_{\textbf{h}'\textbf{q}(0)},\,\mathcal{L}_{k})\bigg]-\eps
\end{aligned}
\end{equation}

where $\varDelta_2>0$ is from Corollary \ref{p} and
$$P_{i,\,k}=\mathbb{P}_{ i}^{\theta_n^{\,\textbf{a}}} \left[\frac{1}{k}H((\theta_n^{\,\textbf{a}})_{\textbf{w},i}, \mathcal{L}_{i+k}^{\varLambda^{\#}})>\varDelta_2\right].$$
\end{lemma}
\begin{proof}
	By Lemma \ref{lemlLowBound},
	if $k,\,n,\,i$ are positive integers with $\hat{n}>t,\,i> C_2k$ and $k\ge K_5(\eps/4)$, for any $\textbf{q}\in\varLambda^{\hat{i}-t}$,  we have 
	\begin{equation}\label{EqQ1}
	\begin{aligned}
	&\frac1k H(B_{\textbf{h}\textbf{q}}(\theta^{\,\textbf{a}}_n),\,\mathcal{L}_{i+k+n}|\mathcal{L}_{i+n})+\frac1k H(B_{\textbf{h}'\textbf{q}}(\theta^{\,\textbf{a}}_n),\,\mathcal{L}_{i+k+n}|\mathcal{L}_{i+n})\\ &\ge
	\mathbb{E}^{\theta^{\,\textbf{a}}_n}_i\bigg[\frac1k H(\bigg\{\gamma^{-(\hat{i}+\hat{n})}A_{\textbf{h}\textbf{q}}\big((\theta^{\,\textbf{a}}_n)_{\textbf{w},i}\big)\bigg\}* m_{\textbf{h}\textbf{q}(0)},\,\mathcal{L}_{k})+\frac1k H(\bigg\{\gamma^{-(\hat{i}+\hat{n})}A_{\textbf{h}'\textbf{q}}\big((\theta^{\,\textbf{a}}_n)_{\textbf{w},i}\big)\bigg\}* m_{\textbf{h}'\textbf{q}(0)},\,\mathcal{L}_{k})\bigg]-\varepsilon/2.
	\end{aligned}
	\end{equation}
	
	For each $\textbf{w}\in\text{supp}(\theta^{\,\textbf{a}}_n)$, let 
	$\varPsi=(\theta^{\,\textbf{a}}_n)_{\textbf{w},i}$. For the case $\frac{1}{k}H(\varPsi, \mathcal{L}_{i+k}^{\varLambda^{\#}})>\varDelta_2,$
	by Lemma \ref{lem:KeyPointLow} we have
	$$\frac1k H\big(\varPsi,\mathcal{L}_{i+k}^{\varLambda^{\#}}\big)
	\le \frac1k H\big( \gamma^{-(\hat{i}+\hat{n})}A_{\textbf{h}\textbf{q}}(\varPsi),\mathcal{L}_{k}\big)+\frac1k H\big( \gamma^{-(\hat{i}+\hat{n})}A_{\textbf{h}'\textbf{q}}(\varPsi),\mathcal{L}_{k}\big)+O(\frac1k),$$
	thus
	 there exist $\textbf{h}^*\in\{\textbf{h},\,\textbf{h}'\}$ such that
	\begin{equation}\label{EqParticalEntropyLow}
	\frac1k H\big( \gamma^{-(\hat{i}+\hat{n})}A_{\textbf{h}^*\textbf{q}}(\varPsi),\mathcal{L}_{k}\big)
	\ge \varDelta_2/4
	\end{equation}
	when $k\ge K_7(\varDelta_2)>0$. 
	
		Let $\eps_4=\min\{\frac{1-\alpha}4, \frac{\varDelta_2}8,\frac1{16b^t}\}$.
	By Theorem \ref{thm:entporous}, let  $m= M_1(\eps_4).$ For any  $k\ge K_1(\eps_4,m)$ and $\hat{i}\ge t+N_1(\eps_4,m,k)$, thus we have
	\begin{equation}\label{eq:entropyporous}
	\nu^{\hat{i}-t}\bigg(\bigg\{\,\textbf{q}\in\varLambda^{\hat{i}-t}\,:\,
	\begin{matrix}
	m_{\textbf{h}\textbf{q}(0)},\,m_{\textbf{h}'\textbf{q}(0)}\text{ are } (1-\varepsilon_4, \eps_4/2, m)\\-\text{entropy porous from scale } 1 \text{ to } k
	\end{matrix}
	\bigg\}\bigg)>\frac12.
	\end{equation}
	Also  Lemma \ref{lem:constantR} implies that $A_{\textbf{h}^*\textbf{q}}(\varPsi) $ is supported in an interval of length $O(b^{-(i+n)})$.  Combining this with Theorem \ref{thm:hochmanentgrow} we have following property.
	
	 There exists $\varDelta_5=\delta(\eps_4,m)>0$ such that for any $k\ge K(\eps_4,\varDelta_5,m)$, if
	\begin{enumerate}
		\item[(1)] $m_{\textbf{h}^*\textbf{q}(0)}$ is $(1-\eps_4,\frac{\eps_4}2,m)$-entropy porous from scales $1$ to $k$;
		\item[(2)] 	$\frac1k H\big( \gamma^{-(\hat{i}+\hat{n})}A_{\textbf{h}^*\textbf{q}}(\varPsi),\mathcal{L}_{k}\big)
		\ge \varDelta_2/4$,
	\end{enumerate}
then
\begin{equation}\label{entropyincrease}
\frac1k H\big(\big\{\gamma^{-(\hat{i}+\hat{n})}A_{\textbf{h}^*\textbf{q}}(\varPsi)\big\}* m_{\textbf{h}^*\textbf{q}(0)},\,\mathcal{L}_{k}\big)\ge\frac1k H(m_{\textbf{h}^*\textbf{q}(0)},\,\mathcal{L}_{k})+\varDelta_5.
\end{equation}
Let $\varDelta_4=\varDelta_5/2$,
$K_6=\max\{K_5(\eps),\,K_7(\varDelta_2),\,K_1(\eps_4,\,m),\,K(\eps_4,\,\varDelta_5,\,m)\}$
and 
$I_1=kC_2(t+N_1(\eps_4,\,m,\,k)).$
 Thus
	Combining (\ref{entropyincrease}) with (\ref{EqQ1}), (\ref{EqParticalEntropyLow}) and (\ref{eq:entropyporous}), our claim holds.
	\end{proof}

	\begin{proof}[The proof of Theorem A]
		Assume that $\alpha<\min\{1,\,\frac{\log b}{\log1/\gamma}\}$.
	Let $p_0$ be the value in Corollary \ref{p}. $\varDelta_4$ is from Lemma \ref{lem:EntropyIncrease}. Denote
	$\eps=\frac{\varDelta_4 p_0}{16b^{2t}}$.
	 By (\ref{eqn:thetamudec}), Lemma \ref{lemTrivialLowBound} and  Lemma \ref{lem:EntropyIncrease}, if $k,\,n$ are positive integers with $n> \max\{C_0(\eps), C_2\}k$, $k\ge \max\{ K_4(\eps),\,K_6(\eps,\varDelta_4,\alpha)\}$ and  $n\ge I_1(\eps,\varDelta_2,\alpha,k)$, then we have
	\begin{equation}\label{large}
	\frac{1}{Cn} H(m_{x_0}, \mathcal{L}_{(C+1)n}|\mathcal{L}_{n})\ge \frac1{Cn}\sum_{i> I_1 }^{Cn-1}\frac1{b^{\hat{i}}}\sum_{\textbf{q}\in\varLambda^{\hat{i}}}\frac1k H(m_{\textbf{q}(0)},\,\mathcal{L}_{k})-4\varepsilon +\frac{\varDelta_4}{b^{2t} } \mathbb{P}_{k,Cn}
	\end{equation}
	where 
	$$\mathbb{P}_{k,Cn}=\mathbb{P}_{ I_1\le i< Cn}^{\theta_n^{\,\textbf{a}}} \left[\frac{1}{k}H((\theta_n^{\,\textbf{a}})_{\textbf{w},i}, \mathcal{L}_{i+k}^{\varLambda^{\#}})>\varDelta_2\right].$$
	Also by Lemma \ref{lem:boundsinm}, when k,n are large enough we have
	\begin{equation}\label{integer}
	\frac1{Cn}\sum_{i> I_1 }^{Cn-1}\frac1{b^{\hat{i}}} \sum_{\textbf{q}\in\varLambda^{\hat{i}}}\frac1k H(m_{\textbf{q}(0)},\,\mathcal{L}_{k})\ge \alpha-\eps.
	\end{equation}
	By Corollary \ref{p}, for $n\in M$ large enough, we have
	\begin{equation}\nonumber
	\mathbb{P}_{ I_1\le i< Cn}^{\theta_n^{\,\textbf{a}}} \left[\frac{1}{k}H((\theta_n^{\,\textbf{a}})_{\textbf{w},i}, \mathcal{L}_{i+k}^{\varLambda^{\#}})>\varDelta_2\right]\ge\frac{p_0}2.
	\end{equation}
	Combining this with (\ref{integer}) and (\ref{large}) we have
	\begin{equation}\label{A}
	\frac{1}{Cn} H(m_{x_0}, \mathcal{L}_{(C+1)n}|\mathcal{L}_{n})\ge \alpha+\eps
	\end{equation}
	for  $n\in M$ large enough. However, as $n\to\infty$, the left hand side of (\ref{A}) converges to $\alpha$, a contradiction!
\end{proof}

\section{Appendix: Proof of Theorem \ref{ledrappier-young}}
 Let us consider the bijection 
$$G:\Sigma\times\mathbb{R}\times [0,1)\to\Sigma\times\mathbb{R}\times [0,1)\quad\quad(\textbf{j},y,x)\to(\sigma\text{j},\frac{y-\phi(\frac{x+j_1}b)}{\gamma},\frac{x+j_1}b)$$
where $\sigma$ is the shift transformation on $\Sigma$, and the projection map
$$\Pi:\Sigma\times [0,1)\to\Sigma\times\mathbb{R}\times [0,1)\quad\quad (\textbf{j},x)\to\big(\textbf{j},S(x,\textbf{j}),x\big).$$
Let $\mu=\Pi(\nu^{\mathbb{Z}_+}\times\mathcal{m})$ and it is easy to see that the measure $\mu$ is invariant and ergodic with respect to $G$. 

For $n\in\mathbb{Z}_+$ define $\mathscr{P}^n=\big\{[\textbf{w}]\times\mathbb{R}\times [0,1):\textbf{w}\in\varLambda^n\big\}$. Let
$\mathscr{P}^n(\textbf{j},y,x)$ be the only element of $\mathscr{P}^n$ that contains $(\textbf{j},y,x)$ and $\mathscr{P}^0:=\Sigma\times\mathbb{R}\times [0,1)$. For any $x\in [0,1)$, let 
$$Y_x:\Sigma\to\Sigma\times\mathbb{R}\times [0,1)\quad\quad\textbf{j}\to\big(\textbf{j},S(x,\textbf{j}),x\big)$$
and 
$$\mu_x:=Y_x(\nu^{\mathbb{Z}_+}).$$

We shall recall a basic property of measures $\mu_x$ in this system.
\begin{lemma}\label{lem:A1}
	For every $(\textbf{j},y,x)\in\Sigma\times\mathbb{R}\times [0,1)$, $n\in\mathbb{N}$ and $r>0$, we have
	\begin{equation}
	\label{MeasureG}
	\mu_x\big(\boldsymbol{B}^T_{(\textbf{j},y,x)}(r)\cap\mathscr{P}^{n+1}(\textbf{j},y,x)\big)=\frac1b\mu_{\frac{x+j_1}b}\bigg(\boldsymbol{B}^T_{G(\textbf{j},y,x)}\bigg(\frac r{\gamma}\bigg)\cap\mathscr{P}^n\big(G(\textbf{j},y,x)\big)\bigg)
	\end{equation}
	where $\boldsymbol{B}^T_{(\textbf{j},y,x)}(r)=\bigg\{\,(\textbf{j}',y',x')\in\Sigma\times\mathbb{R}\times [0,1)\, :\,|y-y'|\le r\,\bigg\}.$
\end{lemma}
\begin{proof}
	We only need to show the case $n=0$ and others are similar. 
	Let us consider the sets $$E=\bigg\{\textbf{j}'\in\Sigma:\big|S(x,\textbf{j}')-y\big|\le r,j'_1=j_1\bigg\}$$
	and
	$$F=\bigg\{\textbf{i}\in\Sigma:\bigg|S(\frac{x+j_1}b,\textbf{i})-\frac{y-\phi(\frac{j_1+x}b)}{\gamma}\bigg|\le \frac r{\gamma}\bigg\}.$$
	By the definition of the function $S$, we have $\sigma E=F$. Thus
	$$\nu^{\mathbb{Z}_+}(E)=\frac1b\nu^{\mathbb{Z}_+}(F)$$
	by the property of Bernoulli measure $\nu^{\mathbb{Z}_+}$, which implies (\ref{MeasureG}) holds.
	
\end{proof}

Let us consider the measurable partition of $\Sigma\times\mathbb{R}\times [0,1)$ defined by
$$\zeta:=\bigg\{\Sigma\times\{y\}\times\{x\}:y\in\mathbb{R},x\in [0,1)\bigg\}.$$
 Then  by the result of Rokhlin \cite{rokhlin1949fundamental}, 
there exists a canonical system of conditional measure.  for $\mu$ almost every $(\textbf{j},y,x)\in\Sigma\times\mathbb{R}\times [0,1)$ there exists conditional measure $\mu_{(\textbf{j},y,x)}\in\mathscr{P}\big(\Sigma\times\mathbb{R}\times [0,1)\big)$ such that

(C.1) $\text{supp}\,\mu_{(\textbf{j},y,x)}\subset\,\zeta(\textbf{j},y,x)$ 

(C.2) $\mu_{(\textbf{j},y,x)}=\mu_{(\textbf{j}',y,x)}$, for all $(\textbf{j}',y,x)\in\zeta(\textbf{j},y,x)$.

(C.3) for every $A\in\mathscr{B}(\Sigma\times\mathbb{R}\times [0,1))$,
\begin{equation}
\label{RokhlinEq1}
\mu(A)=\int_{\Sigma\times\mathbb{R}\times [0,1)}\mu_{(\textbf{j},y,x)}(A)d\mu(\textbf{j},y,x)
\end{equation}
\begin{equation}
\label{RokhlinEq2}
\mu(A)=\int_{ [0,1)}\int_{\Sigma\times\mathbb{R}}\mu_{(\textbf{j},y,x)}(A)d\mu_x(\textbf{j},y,x)dx
\end{equation}
\begin{equation}
\label{RokhlinEq3}
\mu(A)=\int_{ [0,1)}\int_{\mathbb{R}}\mu_{(\textbf{0}_{\infty},y,x)}(A)dm_x(y)dx
\end{equation}
where $\textbf{0}_{\infty}=00\ldots\in\Sigma.$

\subsection{The exact properties of $m_x$}
In this subsection we shall show that $m_x$ is exact dimensional for $\mathcal{m}\,a.e.\,x\in [0,1)$, and give a formula for $dim(m_x)$. The idea of the proof in this part is from \cite{ledrappier1992dimension}. 
Define on $\Sigma\times\mathbb{R}\times [0,1)$ the mapping $g_k$ and g by
$$\quad g_k(\textbf{j},y,x)=-log \frac{\mu_x\bigg(\boldsymbol{B}^T_{(\textbf{j},y,x) } (\gamma^k)\cap\mathscr{P}(\textbf{j},y,x)\bigg) }{\mu_x\big(\boldsymbol{B}^T_{(\textbf{j},y,x) } (\gamma^k)\big)}$$

and 

$$
g(\textbf{j},y,x)=-log\bigg(\,\mu_{(\textbf{j},y,x)}\big(\mathscr{P}(\textbf{j},y,x)\,\big)\bigg) $$
for $k\in\mathbb{Z}_+$.


\begin{lemma}\label{lem:bound}
	$\sup\limits_{x\in [0,1)}\int_{\Sigma\times\mathbb{R}\times\{x\}}\sup\limits_{k\ge1}g_k(\textbf{j},y,x)\,d\mu_x(\textbf{j},y,x)<\infty$
\end{lemma}
\begin{proof}
	Let $ g_{sup}(\textbf{j},y,x)=\sup\limits_{m\ge1}g_m(\textbf{j},y,x)$.  It is easy to see that $g_{sup}(\textbf{j},y,x)$ is a constant when $j_1$, $y$ and $x$ are fixed, which is reasonable to consider the set
	$$E_{R,k}^{x}=\bigg\{\,(\textbf{j},y,x)\in \Sigma\times\mathbb{R}\times\{x\} : g_{sup}(\textbf{j},y,x)> R ,\,j_1=k\bigg\}$$
	and
	\begin{equation}
	\widehat{E_{R,k}^x}=\bigg\{\,y\in\mathbb{R} : \exists\,\textbf{j}\,\text{s.t.}\, g_{sup}(\textbf{j},y,x)> R \,\text{and}\,j_1=k\bigg\}.
	\end{equation}
	for any $R>0$ and $k\in\varLambda.$
	
	So for any $y\in \widehat{E_{R,k}^x}$, exists  $t_y\in\mathbb{Z}_+$ and $\text{j}_y\in\Sigma$ satisfied that $g_{t_y}(\textbf{j}_y,y,x)>R$, which implies that
	\begin{equation}
	\label{A4}
	\mu_x\bigg([k]\times\boldsymbol{B}(y,\gamma^{t_y})\times\{x\}\bigg)\le e^{-R}\mu_x\bigg(\Sigma\times\boldsymbol{B}(y,\gamma^{t_y})\times\{x\}\bigg).
	\end{equation}
	So for any $y\in \widehat{E_{R,k}^x}$ we could find a ball $\boldsymbol{B}(y,\lambda^{t_y})\in\mathbb{R}$ and let $\mathscr{A}_{R,k}^x$ be the union of all the above balls when y take all possible elements in $\widehat{E_{R,k}^x}$. By Besicovitch covering theorem there exists a constant integer $\hat{K}=\hat{K}(\mathbb{R})$ and
	$\mathscr{A}_{j}\subset\mathscr{A}_{R,k}$ for $j=1,2,\ldots \hat{K}$ such that
	all elements in $\mathscr{A}_{j}\,$ are pair disjoint and $\bigcup^{\hat{K}}_{j=1}\mathscr{A}_{j}$ is a cover of $\widehat{E_{R,k}^x}$. Therefore 
	$$\mu_x(E_{R,k}^x)\le\sum\limits^{\hat{K}}_{j=1}\sum\limits_{\boldsymbol{B}\in\mathscr{A}_j}\mu_x\bigg( [k]\times\boldsymbol{B}\times\{x\}\bigg)\le\sum\limits^{\hat{K}}_{j=1}\sum\limits_{\boldsymbol{B}\in\mathscr{A}_j}e^{-R}\mu_x\bigg(\Sigma\times\boldsymbol{B}\times\{x\}\bigg)\le\hat{K}e^{-R}$$ with the disjoint  property of the elements of $\mathscr{A}_{j}$ and (\ref{A4}).
	Let 
	$$E_R^x=\bigg\{\,(\textbf{j},y,x)\in \Sigma\times\mathbb{R}\times{S^1} : g_{sup}(\textbf{j},y,x)> R \bigg\},$$ thus we have
	$$\mu_x(E_R^x)\le b \hat{K}e^{-R}.$$
	By the basic formula in probability theorem, thus
	
	$$\int_{\Sigma\times\mathbb{R}\times\{x\}}g_{sup}(\textbf{j},y,x)\,d\mu_x(\textbf{j},y,x)=\int^{\infty}_0\mu_x(E_{R}^x)dR\le b\hat{K}\int_0^{\infty}e^{-R}dR<\infty.$$
	
\end{proof}
By the standard method we have the following fact.
\begin{lemma}\label{lemH}
	$h:=\int_{\Sigma\times\mathbb{R}\times [0,1)} g(\,\textbf{j},y,x)d\mu(\textbf{j},y,x)<\infty$
\end{lemma}
\begin{proof}
	By (\ref{RokhlinEq3}) and (C.2), we have 
	$$h=\int_{ [0,1)}\int_{\mathbb{R}}\int_{\Sigma}g(\,\textbf{j},y,x)\,d\mu_{(\textbf{0}_{\infty},y,x)}(\textbf{j},y,x)dm_x(y)dx$$
	where $\textbf{0}_{\infty}=00\cdots0\cdots\in\Sigma.$
	Thus we have \begin{equation}\label{entropyh}
	h=\int_{ [0,1)}\int_{\mathbb{R}}\sum_{k\in\varLambda}F\bigg(\mu_{(\textbf{0}_{\infty},y,x)}\bigg([k]\times\{y\}\times\{x\}\bigg)\bigg)dm_x(y)dx
	\end{equation}
	where 
	$$F(s)=s\log\frac1s\quad\quad\forall s\ge0.$$
	Since $\sum_{k\in\varLambda}\mu_{(\textbf{0}_{\infty},y,x)}\big([k]\times\{y\}\times\{x\}\big)=1$, we have 
	$$\sum_{k\in\varLambda}F\bigg(\mu_{(\underline{\textbf{0}}_{\infty},y,x)}\bigg([k]\times\{y\}\times\{x\}\bigg)\bigg)\le\log b.$$
	
	Combining this with (\ref{entropyh}) we have
	$$h\le\log b.$$
	
\end{proof}
\begin{lemma}\label{lemExactProperty}
	For $\mathcal{m}$ a.e. $x\in [0,1)$, the following holds
	
	(1)$m_x$ is exact dimensional;
	
	(2)$dim(m_x)=\frac{log b}{log \gamma}(\frac{h}{log b}-1).$
\end{lemma}
\begin{proof}
	We first show the claim: 
	\begin{equation}
	\label{B4}
	\lim\limits_{n\to\infty}\frac1n\sum\limits_{k=0}^{n-1}g_{n-k}\circ G^{k}(\textbf{j},y,x)=h\quad\mu\, a.e.\,(\textbf{j},y,x)\in\Sigma\times\mathbb{R}\times [0,1).
	\end{equation}
	By Birkhoff ergodic theorem $\lim\limits_{n\to\infty}\frac1n\sum\limits_{k=0}^{n-1}g\circ G^{k}(\textbf{j},y,x)=h\quad\mu\,a.e.\,(t,z,\mathbf{i})\in\Sigma\times\mathbb{R}\times [0,1)$,
	since $g\in L^1(\mu)$ by Lemma \ref{lemH}.
	
	Thus we only need  to prove that 
	$$\limsup\limits_{n\to\infty}\frac1n\sum\limits_{k=0}^{n-1}|g_{n-k}-g|\circ G^{k}(\textbf{j},y,x)=0\quad\mu\, a.e.\,(\textbf{j},y,x)\in\Sigma\times\mathbb{R}\times [0,1).$$
	By lemma \ref{lem:bound} we can define
	$$\varDelta_N(\textbf{j},y,x)=\sup\limits_{k\ge N}|g_{k}-g|(\textbf{j},y,x)$$
	for $\forall N \in\mathbb{Z}_+$. 
	
	For fixed $N$, with lemma \ref{lem:bound} and Birkhoff ergodic theorem, we have
	$$\limsup\limits_{n\to\infty}\frac1n\sum\limits_{k=0}^{n-1}|g_{n-k}-g|\circ G^{k}(\textbf{j},y,x)\le\limsup\limits_{n\to\infty}\frac1n\sum\limits_{k=0}^{n-1}\varDelta_N\circ G^{k}(\textbf{j},y,x)$$
	for 
	$\mu\,a.e.\,(\textbf{j},y,x)\in\Sigma\times\mathbb{R}\times [0,1).$
	Also by Bikhoff ergodic theorem we have
	$$\lim\limits_{n\to\infty}\frac1n\sum\limits_{k=0}^{n-1}\varDelta_N\circ G^{k}(\textbf{j},y,x)=\mathbb{E}\varDelta_N\quad\mu\,a.e.\,(\textbf{j},y,x)\in\Sigma\times\mathbb{R}\times [0,1).$$
	Then for every $N \in\mathbb{Z}_+$,
	$$\limsup\limits_{n\to\infty}\frac1n\sum\limits_{k=0}^{n-1}|g_{n-k}-g|\circ G^{k}(\textbf{j},y,x)\le\mathbb{E}\varDelta_N\quad\mu\,a.e.\,(\textbf{j},y,x)\in\Sigma\times\mathbb{R}\times [0,1).$$
	
	Finally combining $\varDelta_N(\textbf{j},y,x)\le(\sup\limits_{k\ge1}g_{k}+g)(\textbf{j},y,x)$ and 
	$$\lim\limits_{k\to\infty}g_{k}(\textbf{j},y,x)=g(\textbf{j},y,x)\quad\mu\,a.e.\,(\textbf{j},y,x)\in\Sigma\times\mathbb{R}\times [0,1)$$
	by measure differential theorem and (\ref{RokhlinEq2}) , which implies
	$$\lim\limits_{n\to\infty}\mathbf{E}\,\varDelta_{n}=0$$ 
	by Lebesgue dominated convergent theorem.
	So we have shown
	$$\limsup\limits_{n\to\infty}\frac1n\sum\limits_{k=0}^{n-1}|g_{n-k}-g|\circ G^{k}(\textbf{j},y,x)=0\quad\mu\,a.e.\,(\textbf{j},y,x)\in\Sigma\times\mathbb{R}\times [0,1).$$

	Now we shall finish the proof. 
	Without loss of generality we assume $\mu_x[\boldsymbol{B}_{((\textbf{j},y,x)}(1)]=1$ for $\mu\,a.e.\,(\textbf{j},y,x)\in\Sigma\times\mathbb{R}\times [0,1)$.
	Thus we have
	
	\begin{equation}
	\label{B5}
	\mu_x\bigg(\boldsymbol{B}^T_{(\textbf{j},y,x)}(\gamma^n)\bigg)=\prod_{k=0}^{n-1}\frac{\mu_{\textbf{j}_k(x)}\big(\boldsymbol{B}^T_{G^k(\textbf{j},y,x)}(\gamma^{n-k})\big)}{\mu_{\textbf{j}_{k+1}(x)}\big(\boldsymbol{B}^T_{G^{k+1}(\textbf{j},y,x)}(\gamma^{n-k-1})\big)}.
	\end{equation}
	
	for  $\mu\,a.e.\,(\textbf{j},y,x)\in\Sigma\times\mathbb{R}\times [0,1)$.
	(\ref{B5}) implies
	$$\mu_x\bigg(\boldsymbol{B}^T_{(\textbf{j},y,x)}(\gamma^n)\bigg)=\prod_{k=0}^{n-1}\frac{\mu_{\textbf{j}_k(x)}\bigg(\boldsymbol{B}^T_{G^k(\textbf{j},y,x)}(\gamma^{n-k})\bigg)}{b\cdot\mu_{\textbf{j}_k(x)}\bigg(\boldsymbol{B}^T_{G^{k}(\textbf{j},y,x)}(\gamma^{n-k})\cap\mathscr{P}\big(G^{k}(\textbf{j},y,x)\big)\bigg)}.$$
	by lemma \ref{MeasureG}. Thus
	$$ - log \,\mu_x[\boldsymbol{B}^T_{(\textbf{j},y,x)}(\gamma^n)] =n log b+\sum\limits_{k=0}^{n-1}g_{n-k}\circ G^{k}(\textbf{j},y,x).$$
	Combining this with (\ref{B4})  we have 
	$$\lim\limits_{n\to\infty}\frac{\log\mu_x\bigg(\boldsymbol{B}^T_{(\textbf{j},y,x)}(\gamma^n)\bigg)}{log \gamma^n}=\frac{log b}{log \gamma}(\frac{h}{log b}-1)\quad\mu\,a.e.\,(\textbf{j},y,x)\in\Sigma\times\mathbb{R}\times [0,1).$$
	Since $m_x\bigg(\boldsymbol{B}(y,\gamma^n)\bigg)=\mu_x\bigg(\boldsymbol{B}^T_{(\textbf{j},y,x)}(\gamma^n)\bigg)$ for any $\textbf{j}\in\Sigma$, the lemma holds.
\end{proof}
\subsection{The proof of Theorem \ref{ledrappier-young}}
In this subsection, we will use Lemma \ref{lemExactProperty}  to prove Theorem \ref{ledrappier-young} by a standard method. For convenience let
$$\alpha=\frac{log b}{log \gamma}(\frac{h}{log b}-1)$$
and
$$M^{\phi}=M^{\phi}_{\gamma,\,b}=\sup_{\textbf{j}\in\Sigma\cup\varLambda^{\#},\,x\in [0,1]}\big|S'(x,\textbf{j})\big|.$$
\begin{proof}[Proof of Theorem \ref{ledrappier-young}]
	For any $\varepsilon,\,\delta>0$, by Lemma \ref{lemExactProperty}, Egoroff theorem and differentiation theorems for measures, there exists set $E\subset\Sigma\times\mathbb{R}\times [0,1)$  and $r_0\in(0,1)$ such that
	the following holds.
		For all $r\in(0,r_0)$ and $z=(x,y)\in E$, we have
	\begin{enumerate}
		\item[(1)] $\omega(E)>1-\delta$;
		\item[(2)]	$$\omega\bigg(E\cap\boldsymbol{B}(z,r)\bigg)\ge\frac12	\omega\bigg(\boldsymbol{B}(z,r)\bigg);$$		
		\item[(3)] 
		$$r^{\alpha+\varepsilon}\le m_x\bigg(\boldsymbol{B}(y,r)\bigg)\le r^{\alpha-\varepsilon}.$$
		
	\end{enumerate}
	 We first give the upper bound estimate.
	By (C.2) and Rokhlin decomposition of $\omega$ we have 
	\begin{equation}\label{EqAboveMeasure}
	\omega\bigg(\boldsymbol{B}(z,r)\bigg)\le2\omega\bigg(E\cap\boldsymbol{B}(z,r)\bigg)\le\int_{\boldsymbol{B}(x,r)} m_s\bigg(E_s\cap\boldsymbol{B}(y,r)\bigg)ds
	\end{equation}
	where $E_s=\big\{y'\in [0,1):(s,y')\in E\big\}.$
	
	If $m_s\bigg(E_s\cap\boldsymbol{B}(y,r)\bigg)>0$, there exists $y_s\in\mathbb{R}$ such that $(s,y_s)\in E.$ Thus
	$$m_s\bigg(E_s\cap\boldsymbol{B}(y,r)\bigg)\le m_s\bigg(\boldsymbol{B}(y_s,2r)\bigg)\le(2r)^{\alpha-\varepsilon}$$
	by (C.3) for $r<\frac{r_0}2$. Combining this with (\ref{EqAboveMeasure}) we have 
	$$\omega\bigg(\boldsymbol{B}(z,r)\bigg)\le 2r(2r)^{\alpha-\varepsilon}.$$
	Then we have 
	\begin{equation}\label{EqLowMeasureDimension}
	\liminf_{r\to0^+}\frac{\log\omega\bigg(\boldsymbol{B}(z,r)\bigg)}{\log r}\ge1+ \alpha-\varepsilon
	\end{equation}
	for every 
	$z\in E.$
	
	Finally we give the lower bound estimate. By (C.2) and Rokhlin decomposition of $\omega$ we have 
	\begin{equation}\label{EqLowMeasure}
	\omega\bigg(\boldsymbol{B}\bigg(z,\sqrt{2}(M^{\phi}+1)r\bigg)\bigg)\ge\int_{\boldsymbol{B}(x,r)} m_s\bigg(\boldsymbol{B}\big(y,(M^{\phi}+1)r\big)\bigg)ds.
	\end{equation}
	
	For any $x'\in\boldsymbol{B}(x,r)$ and $\textbf{j}\in\Sigma$ such that
	$\big|S(x,\textbf{j})-y\big|\le r,$
	we have 
	$$\big|S(x',\textbf{j})-y\big|\le M^{\phi}r+r.$$
	Thus
	$$m_{x'}\bigg(\boldsymbol{B}\bigg(y,(M^{\phi}+1)r\bigg)\bigg)\ge m_{x}\bigg(\boldsymbol{B}(y,r)\bigg)\ge r^{\alpha+\varepsilon}$$
	by (3). Combining this with (\ref{EqLowMeasure}) we have 
	$$\omega\bigg(\boldsymbol{B}\bigg(z,\sqrt{2}(M^{\phi}+1)r\bigg)\bigg)\ge 2r(r)^{\alpha+\varepsilon}.$$
	Thus 
	\begin{equation}\label{EqAboveMeasureDimension}
	\limsup_{r\to0^+}\frac{\log\omega\bigg(\boldsymbol{B}(z,r)\bigg)}{\log r}\le 1+\alpha+\varepsilon
	\end{equation}
	For all $z\in E.$ Since $\varepsilon,\delta$ can be arbitrarily close to 0, then (\ref{EqLowMeasureDimension}) and (\ref{EqAboveMeasureDimension}) imply
	$$\lim_{r\to0^+}\frac{\log\omega\bigg(\boldsymbol{B}(z,r)\bigg)}{\log r}=1+ \alpha$$
	For $\mu\,a.e.\,z\in [0,1)\times\mathbb{R}.$
\end{proof}

\bibliographystyle{plain}             

\end{document}